\documentclass{mcom-l}

\usepackage{lipsum}
\usepackage{amsfonts}
\usepackage{amssymb}
\usepackage{graphicx}
\usepackage{epstopdf}
\usepackage{algpseudocode}
\usepackage{tensor}
\usepackage{verbatim}
\usepackage{mathtools}
\usepackage{framed}
\usepackage{tikz}
\usetikzlibrary{calc, shapes, angles, quotes, patterns}
\usetikzlibrary{decorations.pathreplacing, cd, patterns.meta}
\usepackage{subcaption}
\usepackage{bm}
\usepackage{esint}
\usepackage{dsfont}
\usepackage{mathrsfs}
\usepackage{aligned-overset}
\usepackage{hyperref}
\usepackage[nameinlink]{cleveref}

\newif\iflinenumbers
\linenumberstrue
\linenumbersfalse

\hypersetup{
	colorlinks = true,
	allcolors = green!50!black,
	urlcolor = red!50!black,
}

\newtheorem{theorem}{Theorem}[section]
\newtheorem{lemma}[theorem]{Lemma}

\theoremstyle{definition}

\theoremstyle{remark}
\newtheorem{remark}[theorem]{Remark}

\numberwithin{equation}{section}

\crefformat{equation}{\textup{#2(#1)#3}}
\crefrangeformat{equation}{\textup{#3(#1)#4--#5(#2)#6}}
\crefmultiformat{equation}{\textup{#2(#1)#3}}{ and \textup{#2(#1)#3}}
{, \textup{#2(#1)#3}}{, and \textup{#2(#1)#3}}
\crefrangemultiformat{equation}{\textup{#3(#1)#4--#5(#2)#6}}%
{ and \textup{#3(#1)#4--#5(#2)#6}}{, \textup{#3(#1)#4--#5(#2)#6}}{, and \textup{#3(#1)#4--#5(#2)#6}}

\Crefformat{equation}{#2Equation~\textup{(#1)}#3}
\Crefrangeformat{equation}{Equations~\textup{#3(#1)#4--#5(#2)#6}}
\Crefmultiformat{equation}{Equations~\textup{#2(#1)#3}}{ and \textup{#2(#1)#3}}
{, \textup{#2(#1)#3}}{, and \textup{#2(#1)#3}}
\Crefrangemultiformat{equation}{Equations~\textup{#3(#1)#4--#5(#2)#6}}%
{ and \textup{#3(#1)#4--#5(#2)#6}}{, \textup{#3(#1)#4--#5(#2)#6}}{, and \textup{#3(#1)#4--#5(#2)#6}}

\crefname{lemma}{Lemma}{Lemmas}
\crefname{theorem}{Theorem}{Theorems}
\crefname{corollary}{Corollary}{Corollaries}
\crefname{figure}{Figure}{Figures}
\Crefname{figure}{Figure}{Figures}
\crefname{table}{Table}{Tables}
\Crefname{table}{Table}{Tables}
\crefname{remark}{Remark}{Remarks}
\crefname{appendix}{appendix}{appendices}
\Crefname{appendix}{Appendix}{Appendices}

\newcommand{\harmonic}[1]{\mathfrak{#1}}
\newcommand{\unitvec}[1]{\hat{#1}}
\renewcommand{\d}[1]{\,\mathrm{d}{#1}}

\newcommand{\dual}[1]{{#1}'}

\usepackage{amsopn}
\let\div\undefined

\DeclareMathOperator{\GN}{GN}
\DeclareMathOperator{\CG}{CG}
\DeclareMathOperator{\DG}{DG}

\DeclareMathOperator{\curl}{curl}
\DeclareMathOperator{\grad}{grad}
\DeclareMathOperator{\hess}{hess}
\DeclareMathOperator{\div}{div}
\DeclareMathOperator{\rot}{rot}

\DeclareMathOperator{\mskw}{mskw}
\DeclareMathOperator{\dee}{d}
\DeclareMathOperator{\from}{\leftarrow}

\DeclareMathOperator{\spann}{span}

\DeclareMathOperator{\Tr}{Tr}
\DeclareMathOperator{\sym}{sym}
\DeclareMathOperator{\image}{im}

\DeclareMathOperator{\subsimplex}{\unlhd}
\DeclareMathOperator{\strictsubsimplex}{\lhd}
\DeclareMathOperator{\notsubsimplex}{\ntrianglelefteq}

\newcommand{\mesh}[0]{\mathcal{T}}
\newcommand{\boundarymesh}[0]{\mathcal{U}}

\iflinenumbers
	\usepackage[mathlines]{lineno}
\fi

\begin{document}

\title[Cohomology of FEM Stokes Complex]{%
  Cohomology of Finite Element Stokes Complexes on Alfeld Splits}


\author[P.~D.~Brubeck]{Pablo D.\ Brubeck}
\address{Mathematical Institute,
University of Oxford,
Oxford, UK}
\email{brubeckmarti@maths.ox.ac.uk}
\thanks{PB acknowledges that this work has received funding through the UKRI Digital
Research Infrastructure Programme through the Science and Technology Facilities
Council's Computational Science Centre for Research Communities (CoSeC)}

\author[Y.~Liang]{Yizhou Liang}
\address{Mathematical Institute,
	University of Oxford,
	Oxford, UK}
\email{yizhou.liang@maths.ox.ac.uk}
\thanks{YL was partly supported through a Royal Society 
University Research Fellowship (URF\textbackslash R1\textbackslash 221398, 
RF\textbackslash ERE\textbackslash 221047).}

\author[C.~Parker]{Charles Parker}
\address{U.S. Naval Research Laboratory,
	4555 Overlook Ave SW, Washington, DC 20375}
\email{charles.w.parker185.ctr@us.navy.mil}
\thanks{CP was supported by an appointment to the NRC Research
Associateship Program at the U.S. Naval Research Laboratory, 
administered by the Fellowships Office of the National Academies of Sciences, 
Engineering, and Medicine. Distribution Statement A.  Approved for public release: distribution is unlimited.}


\date{}

\dedicatory{}

\iflinenumbers
	\linenumbers
\fi

\begin{abstract}
  We show that the cohomology of the finite element Stokes complex consisting 
	of piecewise polynomials spaces on an Alfeld split mesh from 
	Fu, Guzm\'{a}n, \& Neilan 
	(2020, \textit{Math. Comp.}, \textbf{89}, 1059--1091) is 
	isomorphic to the cohomologies of the continuous Stokes and de Rham 
	complexes. 
	We also construct novel ``minimal'' conforming finite element complexes 
	where the $H^1$-conforming space is the lowest-order space from 
	Guzm\'{a}n \& Neilan 
	(2018, \textit{SIAM J. Numer. Anal.}, \textbf{56}, 2826--2844)
	and the $L^2$-conforming space is piecewise constants.
	These minimal complexes also have cohomologies isomorphic to the 
	continuous Stokes and de Rham complexes. We further construct local, 
	bounded, cochain projections for the minimal complexes.
	All the results hold for strongly Lipschitz domains with nontrivial 
	topologies and in the presence of mixed boundary conditions. 
\end{abstract}

\maketitle


\section{Introduction}
\label{sec:intro}

We consider conforming finite element 
discretizations of the Stokes complex with mixed boundary 
conditions. Let $\Omega \subset \mathbb{R}^3$ be a polyhedral domain whose
boundary is partitioned into suitably regular subsets $\Gamma_0$ and $\Gamma_1$,
and define the following spaces:
\begin{subequations}
	\begin{alignat}{2}
		\label{eq:hkd-first-def}
		H^k_{\Gamma_0}(\Omega) &:= 
		\{ \phi \in H^k(\Omega) : D^{\alpha} \phi|_{\Gamma_0} = 0,
		\ \forall |\alpha|  \leq k-1 \}, \qquad & &k \in \mathbb{N}, \\
		H^1_{\Gamma_0}(\Omega; \curl) &:= \{ 
		v \in H^1_{\Gamma_0}(\Omega)^3 : 
		\curl v \in H^1_{\Gamma_0}(\Omega)^3 \}, \qquad & &
	\end{alignat}
\end{subequations}
where we may drop the subscript to denote the case $|\Gamma_0| = 0$ and use the 
subscript ``0'' to denote $|\Gamma_1| = 0$.
These spaces fit into the so-called Stokes complex:
\begin{equation}
	\label{eq:stokes-complex-bcs}
	\begin{tikzcd}
		0 \arrow[r] 
		& H_{\Gamma_0}^2(\Omega) \arrow[r, "\grad"] 
		& H_{\Gamma_0}^1(\Omega; \curl) \arrow[r, "\curl"] 
		& H_{\Gamma_0}^1(\Omega)^3 \arrow[r, "\div"]
		& L^2(\Omega) \arrow[r]
		& 0.
	\end{tikzcd}
\end{equation}
In particular, that \cref{eq:stokes-complex-bcs} is a complex means that
the composition of any two operators is zero (e.g. $\curl \grad = 0$)
and that the image of each operator lies in the succeeding space
(e.g. $\curl H_{\Gamma_0}^1(\curl;\Omega) \subset H_{\Gamma_0}^1(\Omega)^3$). 
Problems involving the spaces in the complex
\cref{eq:stokes-complex-bcs} arise in a variety of 
applications. Many fourth-order problems involve the space $H^2(\Omega)$,
such as the separation of binary alloys \cite{CahnHill58} or 
displacement formulations of strain gradient theory \cite{Mindlin64} 
to name a few. The space $H^1(\curl; \Omega)$ appears in 
displacement formulations of couple stress theory \cite{ParkGao08}. 
The final two spaces most famously appear in incompressible flow.

A conforming
finite element discretization or subcomplex of \cref{eq:stokes-complex-bcs}
is another complex 
\begin{equation}
	\label{eq:stokes-complex-bcs-intro-discrete}
	\begin{tikzcd}
		0 \arrow[r] 
		& V^{0, h}_{\Gamma_0} \arrow[r, "\grad"] 
		& V^{1, h}_{\Gamma_0} \arrow[r, "\curl"] 
		& V^{2, h}_{\Gamma_0} \arrow[r, "\div"]
		& V^{3, h}_{\Gamma_0} \arrow[r]
		& 0,
	\end{tikzcd}
\end{equation}
where each space $V^{k, h}_{\Gamma_0}$ is a conforming finite element subspace 
of the corresponding space in \cref{eq:stokes-complex-bcs}.
Choosing finite element discretizations from an underlying complex 
\cref{eq:stokes-complex-bcs-intro-discrete} can offer many benefits. 
In velocity-pressure formulations of incompressible flow, 
taking $V^{2, h}_{\Gamma_0}$ and $V^{3, h}_{\Gamma_0}$ from 
\cref{eq:stokes-complex-bcs-intro-discrete} lead to mass 
conserving  and ``pressure-robust" discretizations 
\cite{JohnLinkeMerdonNeilanRebholz17} and schemes with uniform stability 
properties for time-dependent or singularly perturbed flows 
\cite{MardalTaiWinther02,TaiWinther06}. All four spaces from 
\cref{eq:stokes-complex-bcs-intro-discrete} can also be used to construct 
an energy and enstrophy stable scheme for the 
incompressible Navier-Stokes equations \cite[p. 168 eq. (10.42)]{Andrews25}.
The complex \cref{eq:stokes-complex-bcs-intro-discrete} also encodes
information that is useful for preconditioning parameter-dependent problems. 
For example, consider the weighted bilinear forms for 
$\alpha, \beta \in \mathbb{R}_+$:
\begin{subequations}
	\label{eq:riesz-maps}
	\begin{alignat}{2}
		&(\grad u, \alpha \grad v)_{L^2(\Omega)} 
		+ (\grad \curl u, \beta \grad \curl v)_{L^2(\Omega)}
		\qquad & &\forall u, v \in V_{\Gamma_0}^{1, h}, \\
		&(\grad u, \alpha \grad v)_{L^2(\Omega)} 
		+ (\div u, \beta \div v)_{L^2(\Omega)}
		\qquad & &\forall u, v \in V_{\Gamma_0}^{2, h},	
	\end{alignat}	
\end{subequations}
which arise in couple stress theory \cite{ParkGao08} and 
augmented Lagrangian preconditioning for incompressible flow 
\cite{BenziOlshankskii06,FarrellMitchellScottWechsung21}. 
Constructing preconditioners that are robust 
in the parameters $\alpha$ and $\beta$ typically requires knowledge 
of the kernel of $\curl : V_{\Gamma_0}^{1, h} \to V_{\Gamma_0}^{2, h}$
and $\div : V_{\Gamma_0}^{2, h} \to V_{\Gamma_0}^{3, h}$ \cite{LeeWuXuZikatanov07,Schoberl99},
which is precisely encoded in the algebraic structure of the complex 
\cref{eq:stokes-complex-bcs-intro-discrete}.

More specifically, the \textit{cohomology} of a complex consists of the kernel
of an operator modulo the range of the previous operator. For example,
the first and second cohomology of the discrete complex 
\cref{eq:stokes-complex-bcs-intro-discrete} are
\begin{align*}
	\mathfrak{H}^{1, h}_{\Gamma_0} := \frac{ \ker(\curl : V_{\Gamma_0}^{1, h} \to V_{\Gamma_0}^{2, h}) }{ 
		\image(\grad : V_{\Gamma_0}^{0, h} \to V_{\Gamma_0}^{1, h}) }
	\quad \text{and} \quad 
	\mathfrak{H}^{2, h}_{\Gamma_0} := \frac{ \ker(\div : V_{\Gamma_0}^{2, h} \to V_{\Gamma_0}^{3, h}) }{ 
		\image(\curl : V_{\Gamma_0}^{1, h} \to V_{\Gamma_0}^{2, h}) }.
\end{align*}
Characterizing the cohomology is thus crucial for constructing robust 
preconditioners for \cref{eq:riesz-maps}. The cohomology 
also plays a key role in finite element exterior calculus (FEEC),
including the well-posedness of the Hodge-Laplace problems associated with 
\cref{eq:stokes-complex-bcs-intro-discrete} 
(see e.g. \cite[Chapter 5.2]{Arnold18}) which includes incompressible flow
as well as other mixed problems involving the spaces and operators in 
\cref{eq:stokes-complex-bcs-intro-discrete}. More generally, 
the properties of ``simpler" complexes like
\cref{eq:stokes-complex-bcs-intro-discrete,eq:stokes-complex-bcs}, 
including cohomology, are critical in understanding properties of 
more complicated sequences via the Bernstein--Gelfand--Gelfand construction
\cite{ArnoldHu21,CapHu23,CapHu24}. For a recent review on the importance 
of cohomology in a variety of applications, we refer to \cite{Hu26}.
Nevertheless, the cohomology of any conforming finite element subcomplex
of \cref{eq:stokes-complex-bcs} on nontrivial domains and/or with mixed 
boundary conditions seems to not have been addressed in the literature.

Our first main result shows that if the finite element spaces 
in \cref{eq:stokes-complex-bcs-intro-discrete} are chosen to be the 
Alfeld-split macroelements in 
\cite{FuGuzmanNeilan20}, 
then the cohomology of the discrete complex
\cref{eq:stokes-complex-bcs-intro-discrete} is isomorphic to the cohomology
of the continuous complex \cref{eq:stokes-complex-bcs}. This result extends
\cite[Theorem 5.1]{FuGuzmanNeilan20} to the case of nontrivial domains 
and to the case of mixed boundary conditions. The main technique is 
applying the recent framework of \cite{HuLiangLin25} and modifying the final 
steps to account for boundary conditions. The second set of results show that,
in a certain sense, the Alfeld-split macroelement Stokes complex analog of 
Whitney forms 
\cite{Bossavit88a,Bossavit88b,Whitney57}
and lowest-order complete polynomial complexes can be constructed 
to form ``minimal'' subcomplexes of \cref{eq:stokes-complex-bcs-intro-discrete} 
with the same cohomology structure. The final two spaces in these subcomplexes
consist of low-order $H^1(\Omega)^3$-conforming element from 
\cite[Section 4]{GuzmanNeilan18} and piecewise constants.
For the first two spaces in the subcomplexes, we obtain novel 
$H^2(\Omega)$-conforming and $H^1(\curl; \Omega)$-conforming finite elements
whose dimension is, in a certain sense, ``minimal''.

\section{Summary of main results and outline}
\label{sec:main-results}

We assume that the domain $\Omega \subset \mathbb{R}^3$ and boundary partition
$\partial \Omega = \bar{\Gamma}_0 \cup \bar{\Gamma}_N$ satisfy the assumptions 
in \cite{PaulySchomburg22}: $\Omega$ is a bounded strongly 
Lipschitz domain and $\Gamma_0$ and $\Gamma_1$ are strongly Lipschitz subsets,
so that $\Omega$ and $\Gamma_0$ form a strong Lipschitz pair 
in the sense of \cite[Definition 2.11]{BauerPaulySchomburg19}.
We further assume that $\partial \Omega$ and $\Gamma_0$ are both the
union of a finite number of polygons.

We rewrite the Stokes complex \cref{eq:stokes-complex-bcs} in standard
FEEC notation as follows:
\begin{equation}
	\label{eq:stokes-complex-bcs-feec}
	\begin{tikzcd}
		0 \arrow[r] 
		& V_{\Gamma_0}^0 \arrow[r, "\dee^0"] 
		& V_{\Gamma_0}^1 \arrow[r, "\dee^1"] 
		& V_{\Gamma_0}^2 \arrow[r, "\dee^2"]
		& V_{\Gamma_0}^3 \arrow[r]
		& 0,
	\end{tikzcd}
\end{equation}
where 
\begin{alignat*}{4}
	V_{\Gamma_0}^0 &:= H^2_{\Gamma_0}(\Omega), \qquad & 
	V_{\Gamma_0}^1 &:= H^1_{\Gamma_0}(\curl; \Omega), \qquad & 
	V_{\Gamma_0}^2 &:= H^1_{\Gamma_0}(\Omega)^3, & \qquad 
	V_{\Gamma_0}^3 &:= L^2(\Omega), \\
	\dee^0 &:= \grad, \qquad & 
	\dee^1 &:= \curl, \qquad & 
	\dee^2 &:= \div, \qquad & 
	\dee^3 &:= 0.
\end{alignat*}
As with the Sobolev spaces, we may 
drop the subscript ``$D$'' if $|\Gamma_0| = 0$ or use the 
subscript ``0'' if $|\Gamma_1| = 0$. The Stokes complex 
\cref{eq:stokes-complex-bcs-feec} may be viewed as 
a smoother subcomplex of the standard de Rham complex:
\begin{equation}
	\label{eq:de-rham-complex-bcs-feec}
	\begin{tikzcd}
		0 \arrow[r] 
		& W_{\Gamma_0}^0 \arrow[r, "\dee^0"] 
		& W_{\Gamma_0}^1 \arrow[r, "\dee^1"] 
		& W_{\Gamma_0}^2 \arrow[r, "\dee^2"]
		& W_{\Gamma_0}^3 \arrow[r]
		& 0,
	\end{tikzcd}
\end{equation}
where 
\begin{alignat*}{4}
	W_{\Gamma_0}^0 &:= H^1_{\Gamma_0}(\Omega), \quad 
	& W_{\Gamma_0}^1 &:= H_{\Gamma_0}(\curl; \Omega), \quad 
	& W_{\Gamma_0}^2 &:= H_{\Gamma_0}(\div; \Omega), \quad
	& W_{\Gamma_0}^3 &:= L^2(\Omega).
\end{alignat*}
Here, $H_{\Gamma_0}(\curl; \Omega)$, respectively $H_{\Gamma_0}(\div; \Omega)$, is the space of 
$L^2(\Omega)^3$ vector fields whose $\curl$, respectively $\div$, is also 
square integrable and whose tangential, respectively normal, trace vanishes 
on $\Gamma_0$.

The $k$th-cohomology, or harmonic form, of the Stokes complex is 
defined by
\begin{align}
	\label{eq:harmonic-forms-bcs}
	\harmonic{H}_{\Gamma_0}^k 
	:= \frac{ \ker(\dee^k : V_{\Gamma_0}^k \to V_{\Gamma_0}^{k+1}) }{ 
		\image(\dee^{k-1} : V_{\Gamma_0}^{k-1} \to V_{\Gamma_0}^k) },
	\qquad k \in 0:3,
\end{align}
where we adopt the convention that any space or operator with form index 
$k < 0$ or $k > 3$ is the trivial space or operator, 
and $m:n := \{ m, m+1,\dots, n-1, n\}$
for nonnegative integers $m, n$. Theorem 5.12\footnote{
	In \cite{PaulySchomburg22}, $H^k_{\Gamma_0}(\Omega)$ is defined to be the 
	completion of smooth functions whose support is disjoint from
	$\Gamma_0$, which is equivalent to \cref{eq:hkd-first-def} thanks to
	\cite[Theorem 8.7 (iv)]{Brewster14} on noting that $\bar{\Gamma}_0$
	is a closed 2-Ahlfors regular set.
} of \cite{PaulySchomburg22}
shows that the harmonic forms of the Stokes complex and the de Rham 
complex are finite and isomorphic:
\begin{align}
	\label{eq:harmonic-forms-same-dim-diff-smoothness}
	\dim \harmonic{H}_{\Gamma_0}^k = 
	\dim \frac{ \ker(\dee^k : W_{\Gamma_0}^k \to W_{\Gamma_0}^{k+1}) }{ 
		\image(\dee^{k-1} : W_{\Gamma_0}^{k-1} \to W_{\Gamma_0}^k) } 
	\qquad \forall k \in 0:3.
\end{align}

The dimensions of the harmonic forms correspond to relative Betti numbers.
For $k \in 0:3$, let $b_k(\Omega, \Gamma_0)$ be the $k$-th relative Betti 
number defined as the dimension of the $k$-th singular homology group of 
$\Omega$ relative to $\Gamma_0$
(see e.g. \cite[pp. 108 \& 115]{Hatcher02}). Applying 
\cite[Theorem 5.3 and eq. (5.28)]{GoldshteinMitreaMitrea11}
to \cref{eq:harmonic-forms-same-dim-diff-smoothness} 
then shows that 
\begin{align}
	\label{eq:harmonic-forms-dim}
	\dim \harmonic{H}_{\Gamma_0}^k = b_k(\Omega, \Gamma_0) \qquad \forall k \in 0:3.
\end{align}
Note that $b_k(\Omega, \emptyset)$ corresponds to the usual Betti number 
of $\Omega$; e.g. $b_0(\Omega)$ is the number of connected components 
of $\Omega$. We also have the duality relation 
$b_{k}(\Omega, \Gamma_0) = b_{3-k}(\Omega, \Gamma_1)$ for $k \in 0:3$
thanks to \cite[Corollary 5.4]{GoldshteinMitreaMitrea11}. Combining this 
with \cite[eq. (5.10)]{GoldshteinMitreaMitrea11}, we have 
\begin{align*}
	b_0(\Omega, \Gamma_0) = \begin{cases}
		1 & \text{if } |\Gamma_0| = 0, \\
		0 & \text{otherwise},
	\end{cases}
	\quad \text{and} \quad 
	b_3(\Omega, \Gamma_0) = \begin{cases}
		1 & \text{if } |\Gamma_1| = 0, \\
		0 & \text{otherwise}.
	\end{cases}
\end{align*}
Unfortunately, $b_{1}(\Omega, \Gamma_0)$ and $b_2(\Omega, \Gamma_0)$ do not 
have such simple expressions.

\subsection{Cohomology of full polynomials spaces on Alfeld splits}

We first recall a finite element Stokes complex on Alfeld-split meshes
from \cite{FuGuzmanNeilan20}.
For a collection of tetrahedra $\mathcal{S}$ forming a conforming mesh 
of an open domain $\mathcal{O}$, we define the following spaces of polynomials 
for $p \in \mathbb{N}_0$:
\begin{align*}
	\DG^p(\mathcal{S}) &:= \{ v \in L^2(\mathcal{O}) : 
	v|_{K} \in \mathcal{P}_p(K) \ \forall K \in \mathcal{S} \}
	\ \text{and} \   
	\CG^p(\mathcal{S}) := \DG^p(\mathcal{S}) \cap C(\bar{\mathcal{O}}).
\end{align*}
Given a tetrahedron $K$, let $K_A$ denote 
the collection of four tetrahedra in the Alfeld split of $K$ formed by 
connecting the four vertices of $K$ to the barycenter of $K$,
and for $p \geq 3$, define the following local Alfeld-split macroelement spaces:
\begin{subequations}
	\label{eq:local-full-alfeld}
	\begin{alignat}{2}
		V^{0, h}(K_A) &:= \{v \in \CG^{p+2}(K_A) 
		: \grad v \in \CG^{p+1}(K_A)^3\},\\
		V^{1, h}(K_A) &:= \{v \in \CG^{p+1}(K_A)^3 
		: \curl v \in \CG^{p}(K_A)^3\},\\
		V^{2, h}(K_A) &:= \CG^p(K_A)^3, \\
		V^{3, h}(K_A) &:= \DG^{p-1}(K_A).
	\end{alignat}	
\end{subequations}
These local spaces form a complex 
\begin{equation}
	\label{eq:stokes-complex-local-full-alfeld}
	\begin{tikzcd}
		\mathbb{R} \arrow[r, "\subset"] 
		& V^{0, h}(K_A) \arrow[r, "\dee^0"] 
		& V^{1, h}(K_A) \arrow[r, "\dee^1"] 
		& V^{2, h}(K_A) \arrow[r, "\dee^2"]
		& V^{3, h}(K_A) \arrow[r]
		& 0,
	\end{tikzcd}
\end{equation}
and the complex is exact for all 
$p \geq 3$ \cite[Theorem 4.7]{FuGuzmanNeilan20}, which means 
that the range of each operator is the kernel of the succeeding operator; 
i.e., all harmonic forms are trivial.

The local spaces \cref{eq:local-full-alfeld} may be extended to a multi-element
mesh in a typical finite element fashion provided that extra vertex smoothness 
is imposed on first two spaces. In particular, let $\mesh$ be a 
conforming, shape regular simplicial mesh of $\Omega$ such that 
there exists a submesh of triangles $\boundarymesh$ that forms a conforming 
mesh of $\Gamma_0$. Additionally, we define the global ambient spaces by
\begin{align}
	\label{eq:global-ambient-spaces}
	\mathbb{X}^0 := \mathbb{X}^3 := \mathbb{R} 
	\quad \text{and} \quad \mathbb{X}^1 := \mathbb{X}^2 := \mathbb{R}^3.
\end{align}
\begin{subequations}
	\label{eq:global-full-alfeld}
	Then, for $k \in 0:1$, the global spaces are given by 
	\begin{multline}
		V^{k, h} := \{ v \in C(\Omega) \otimes \mathbb{X}^k : 
		\dee^k v \in C(\Omega) \otimes \mathbb{X}^{k+1}, \
		\text{$v$ is $C^{2-k}$ at $\Delta_0(\mesh)$} \\
		\text{ and } v|_{K} \in V^{k, h}(K_A) \ \forall K \in \mesh\},
	\end{multline}
	where $\Delta_{\ell}(\mathcal{S})$ denotes the set of all
	$\ell$-dimensional subsimplices of a 
	collection of $d$-dimensional simplices $\mathcal{S}$ with $\ell \in 0:d$ 
	(e.g. $\mesh = \Delta_3(\mesh)$). 
	For $k \in 2:3$, the global spaces are simply
	\begin{align}
		V^{k, h} := \{ v \in C^{2-k}(\Omega) \otimes \mathbb{X}^{k} : 
		v|_{K} \in V^{k, h}(K_A) \ \forall K \in \mesh \}, 
	\end{align}
	where $C^{-1}(\Omega) := L^2(\Omega)$.
\end{subequations}
The corresponding spaces incorporating the boundary conditions are 
\begin{align}
	\label{eq:global-full-alfeld-bcs}
	V^{k, h}_{\Gamma_0} := V^{k, h} \cap V_{\Gamma_0}^k,
\end{align}
which form a conforming subcomplex of \cref{eq:stokes-complex-bcs-feec}:
\begin{equation}
	\label{eq:stokes-complex-full-alfeld-bcs}
	\begin{tikzcd}
		0 \arrow[r] 
		& V^{0, h}_{\Gamma_0} \arrow[r, "\dee^0"] 
		& V^{1, h}_{\Gamma_0} \arrow[r, "\dee^1"] 
		& V^{2, h}_{\Gamma_0} \arrow[r, "\dee^2"]
		& V^{3, h}_{\Gamma_0} \arrow[r]
		& 0.
	\end{tikzcd}
\end{equation}

Our first main result shows that the discrete harmonic forms 
\begin{align}
	\label{eq:harmonic-forms-full-alfeld-bcs}
	\harmonic{H}^{k, h}_{\Gamma_0} := \frac{\ker(\dee^k : V_{\Gamma_0}^{k, h} \to V_{\Gamma_0}^{k+1, h})}{ 
		\image(\dee^{k-1} : V_{\Gamma_0}^{k-1, h} \to V_{\Gamma_0}^{k, h})}, \qquad 
	k \in 0:3,
\end{align}
have the same dimension as the 
continuous harmonic forms in \cref{eq:harmonic-forms-bcs}.

\begin{theorem}
	\label{thm:cohomology-full-alfeld-bcs}
	For all $p \geq 3$, the complexes \cref{eq:stokes-complex-bcs-feec}
	and \cref{eq:stokes-complex-full-alfeld-bcs}
	have isomorphic cohomologies: $\dim \harmonic{H}^{k, h}_{\Gamma_0} 
	= \dim \harmonic{H}^{k}_{\Gamma_0} = b_k(\Omega, \Gamma_0)$ for all $k \in 0:3$. 
\end{theorem}

\noindent The proof of \cref{thm:cohomology-full-alfeld-bcs},
appearing below in \cref{sec:proof-cohomology-full-spaces-bcs}, applies the 
general framework of \cite{HuLiangLin25} and extends it to the case of 
mixed boundary conditions 
for the particular discrete complex
\cref{eq:stokes-complex-full-alfeld-bcs}. We note 
that \cref{thm:cohomology-full-alfeld-bcs} is the extension of 
\cite[Theorem 5.1]{FuGuzmanNeilan20} to nontrivial domains with mixed 
boundary conditions. 

\subsection{``Minimal'' conforming complexes}
\label{sec:minimal-complexes-results}

We now seek a ``minimal'' conforming finite element 
subcomplex of the Stokes complex \cref{eq:stokes-complex-bcs-feec}
\begin{equation}
	\label{eq:stokes-complex-reduced-alfeld-bcs}
	\begin{tikzcd}
		0 \arrow[r] 
		& \tilde{V}^{0, h}_{\Gamma_0} \arrow[r, "\dee^0"] 
		& \tilde{V}^{1, h}_{\Gamma_0} \arrow[r, "\dee^1"] 
		& \tilde{V}^{2, h}_{\Gamma_0} \arrow[r, "\dee^2"]
		& \tilde{V}^{3, h}_{\Gamma_0} \arrow[r]
		& 0
	\end{tikzcd}
\end{equation}
whose cohomology is isomorphic to that of \cref{eq:stokes-complex-bcs-feec}.
We restrict ourselves to subcomplexes of 
\cref{eq:stokes-complex-reduced-alfeld-bcs} to stay within the setting 
of Alfeld-split meshes; see \cite[Proposition 27]{ChristiansenHu18} for 
spaces on other types of splits and \cite[p. 343]{ChristiansenHu18} for 
minimal two-dimensional elements.
One of the key properties of the full polynomial spaces 
\cref{eq:global-full-alfeld} that is crucial in the proof of 
\cref{thm:cohomology-full-alfeld-bcs} is that one may choose
degrees of freedom for the spaces $V^{k, h}$ 
to include the following linear functionals:
\begin{subequations}
	\label{eq:whitney-dofs-currents}
	\begin{alignat}{2}
		C(\bar{\Omega}) \otimes \mathbb{X}^0 \ni v 
		&\mapsto \mathcal{I}_{z}^0(v) := v(z) \qquad 
		& &\forall z \in \Delta_0(\mesh), \\
		C(\bar{\Omega}) \otimes \mathbb{X}^1 \ni v 
		&\mapsto \mathcal{I}_{e}^1(v) := 
		\int_{e} v \cdot \unitvec{t}_e \d{s} \qquad 
		& &\forall e \in \Delta_1(\mesh), \\
		C(\bar{\Omega}) \otimes \mathbb{X}^2 \ni v 
		&\mapsto \mathcal{I}_{f}^2(v) := 
		\int_{f} v \cdot \unitvec{n}_f \d{s} \qquad 
		& &\forall f \in \Delta_2(\mesh), \\
		C(\bar{\Omega}) \otimes \mathbb{X}^3 \ni v 
		&\mapsto \mathcal{I}_{K}^3(v) := \int_{K} v \d{x} \qquad
		& &\forall K \in \Delta_3(\mesh),
	\end{alignat}
\end{subequations}
where $\unitvec{n}_f$ and $\unitvec{t}_e$ are unit normal and tangent vectors 
with a fixed global orientation. Note that \cref{eq:whitney-dofs-currents} 
are simply
the canonical set of degrees of freedom for the Whitney forms 
\cite{Bossavit88a,Bossavit88b,Whitney57},
the lowest-order conforming discretization of the de Rham complex 
\cref{eq:de-rham-complex-bcs-feec}.

Our starting point is then to choose 
\begin{align}
	\label{eq:min-l2-space}
	\tilde{V}^{3, h}(K_A) := \DG^0(K), \quad  
	\tilde{V}^{3, h} := \DG^0(\mesh), \quad \text{and} \quad 
	\tilde{V}^{3, h}_{\Gamma_0} := \tilde{V}^{3, h} \cap V^3_{\Gamma_0},
\end{align}
the Whitney forms of index 3.
Moving one space to the left in the complex 
\cref{eq:stokes-complex-reduced-alfeld-bcs},
we seek an $H^1(\Omega)^3$-conforming finite element space whose divergence 
lies in $\DG^0(\mesh)$ and is large enough to be equipped 
with the degrees of freedom in \cref{eq:whitney-dofs-currents}.
The Guzm\'{a}n-Neilan element \cite[Section 4]{GuzmanNeilan18} exactly 
meets these requirements and is defined as follows. For $f \in \Delta_2(K)$, 
let $b_f \in \mathcal{P}_3(K)$ be the face bubble function satisfying 
$b_f|_{\partial K \setminus f} = 0$, normalized so that 
$\int_f b_f \d{s} = |f|$. Henceforth, $b_f$ is also used to denote the 
restriction $b_f|_f$. Moreover, for $K \in \mesh$, let 
$S_{K} : \CG^{3}(K_A)^3 \to \CG^{3}(K_A)^3$ 
be any fixed linear operator satisfying the following
for all $v \in \CG^{3}(K_A)^3$:
\begin{align}
	\label{eq:generic-gn-extension-operator}
	S_{K} v|_{\partial K} 
	= v|_{\partial K}, 
	\quad 
	\div S_{K} v \in \mathcal{P}_0(K),  
	\ \ \text{and} \ \
	| S_K v |_{H^{\ell}(K)} 
	\leq C_{S} |v |_{H^{\ell}(K)}, \ \ell \in 0:1,
\end{align}
where $C_{S} > 0$ is independent of $K$. Then, the local and global 
Guzm\'{a}n-Neilan spaces are give by
\begin{align}
	\label{eq:gn-local}
	\GN(K_A) &:= \mathcal{P}_1(K)^3 
	\oplus 
	\spann\{ S_{K}( b_f \unitvec{n}_f) : f \in \Delta_2(K)  \}
	\qquad \forall K \in \mesh, \\
	\label{gn:global}
	\GN(\mesh) &:= \{v \in C(\Omega)^3 : 
	v|_{K} \in \GN(K_A) \ \forall K \in \mesh  \}.
\end{align}
In particular, $\CG^1(\mesh)^3 \subset \GN(\mesh) \subset V^{1, h}$,
$\div \GN(\mesh) \subseteq \DG^0(\mesh)$, and a global set of 
degrees of freedom are given by \cite[Lemma 4.3]{GuzmanNeilan18}: 
\begin{align}
	\label{eq:gn-dofs}
	v(z) \qquad \forall z \in \Delta_0(\mesh) 
	\quad \text{and} \quad
	\int_f v \cdot n_f \d{s} \qquad \forall f \in \Delta_2(\mesh).
\end{align}

\begin{remark}
	The definition of $\GN(K_A)$ in \cite[Section 4]{GuzmanNeilan18} used a 
	particular choice of $S_K$; however, the only properties used in the 
	analysis are \cref{eq:generic-gn-extension-operator}. Thus, we
	shall use results from \cite{GuzmanNeilan18} for the more generic space 
	here.
\end{remark}

The choice 
\begin{align}
	\label{eq:min-h1-space}
	\tilde{V}^{2, h}(K_A) := \GN(K_A), \quad  
	\tilde{V}^{2, h} := \GN(\mesh), \quad \text{and} \quad 
	\tilde{V}^{2, h}_{\Gamma_0} := \tilde{V}^{2, h} \cap V^2_{\Gamma_0}
\end{align}
is then ``minimal", 
as one typically requires vertex degrees of freedom as in \cref{eq:gn-dofs} 
to ensure continuity. The construction of the remaining spaces 
$\tilde{V}_{\Gamma_0}^{0, h}$ and $\tilde{V}_{\Gamma_0}^{1, h}$ will be detailed in 
\cref{sec:gn-complex-construction} below. The main result is the following.

\begin{theorem}
	\label{thm:cohomology-reduced-alfeld-bcs}
	Let $\tilde{V}^{k, h}(K_A)$ and $\tilde{V}^{k, h}_{\Gamma_0}$ for $k \in 2:3$
	be given by \cref{eq:min-l2-space,eq:min-h1-space}.
	For each $K \in \mesh$ and $k \in 0:1$,
	there exists $\tilde{V}^{k, h}(K_A) \subset V^{k, h}(K_A)$,
	such that
	$\mathcal{P}_{3-k}(K) \otimes \mathbb{X}^k \subset \tilde{V}^{k, h}(K_A)$ 
	and the local complex 
	\begin{equation}
		\label{eq:stokes-complex-reduced-alfeld-element}
		\begin{tikzcd}
			\mathbb{R} \arrow[r, "\subset"] 
			& \tilde{V}^{0, h}(K_A) \arrow[r, "\dee^0"] 
			& \tilde{V}^{1, h}(K_A) \arrow[r, "\dee^1"] 
			& \tilde{V}^{2, h}(K_A) \arrow[r, "\dee^2"]
			& \tilde{V}^{3, h}(K_A) \arrow[r]
			& 0
		\end{tikzcd}
	\end{equation}
	is exact. Moreover, if we define for $k \in 0:1$
	\begin{align}
		\label{eq:min-h2h1curl-space}
		\tilde{V}^{k, h} := \{ v \in V^{k, h} : 
		v|_{K} \in \tilde{V}^{k, h}(K_A) 
		\ \forall K \in \mesh \} 
		\quad \text{and} \quad 
		\tilde{V}^{k, h}_{\Gamma_0} := \tilde{V}^{k, h} \cap V^k_{\Gamma_0},
	\end{align}
	then the global complexes \cref{eq:stokes-complex-bcs-feec} and 
	\cref{eq:stokes-complex-reduced-alfeld-bcs}
	have isomorphic cohomologies.
\end{theorem}
\noindent The proof of \cref{thm:cohomology-reduced-alfeld-bcs} appears in 
\cref{sec:proof-cohomology-reduced-alfeld-bcs} below. 
We will also show below in \cref{lem:walkington-dofs-unisolvent,%
	lem:h1curl-local-unisolvence} that 
$\tilde{V}^{0, h}$ may be equipped with the degrees of freedom 
\begin{align}
	\label{eq:walkington-dofs}
	D^{\alpha} v(z) \qquad \forall |\alpha| \leq 2, 
	\ \forall z \in \Delta_0(\mesh),
\end{align}
while $\tilde{V}^{1, h}$ may be equipped with
\begin{subequations}
	\label{eq:h1curl-dofs}
	\begin{alignat}{2}
		\label{eq:h1curl-dofs-vertices}
		&D^{\alpha} v(z) \qquad & &
		\forall|\alpha| \leq 1, \ \forall z \in \Delta_0(\mesh), \\
		\label{eq:h1curl-dofs-edge-tangents}
		&\int_{e} v \cdot \unitvec{t}_e \d{s} \qquad & &
		\forall e \in \Delta_1(\mesh).
	\end{alignat}
\end{subequations}
In view of the $C^{2-k}$-continuity imposed at the mesh vertices 
of elements in $V^{k, h}$, $k \in 0:1$, we see that 
$\tilde{V}^{k, h}$ are then ``minimal'' subspaces of $V^{k, h}$
whose degrees of freedom can be chosen to include 
\cref{eq:whitney-dofs-currents} and satisfy 
$\mathcal{P}_{3-k}(K) \otimes \mathbb{X}^k \subset \tilde{V}^{k, h}(K_A)$.

In view of this property, the local complex 
\cref{eq:stokes-complex-reduced-alfeld-element} bears 
resemblance to the complex of complete polynomials used to discretize 
the de Rham complex \cref{eq:de-rham-complex-bcs-feec}. 
Locally, the complete polynomial complex reads for $K \in \mesh$:
\begin{equation}
	\label{eq:de-rham-complex-complete-poly-element}
	\begin{tikzcd}
		\mathbb{R} \arrow[r, "\subset"] 
		& \mathcal{P}_{3}(K) \arrow[r, "\dee^0"] 
		& \mathcal{P}_{2}(K)^3 \arrow[r, "\dee^1"] 
		& \mathcal{P}_{1}(K)^3 \arrow[r, "\dee^2"]
		& \mathcal{P}_{0}(K) \arrow[r]
		& 0.
	\end{tikzcd}
\end{equation}
In fact, \cref{thm:cohomology-reduced-alfeld-bcs} shows that 
\cref{eq:de-rham-complex-complete-poly-element} is a subcomplex of 
\cref{eq:stokes-complex-reduced-alfeld-element}. The additional complexities 
of the spaces in \cref{eq:stokes-complex-reduced-alfeld-element} only
arise due to the additional global continuity imposed by being 
conforming subspaces of $V_{\Gamma_0}^{k}$ rather than $W_{\Gamma_0}^{k}$. Thus, the global
complex \cref{eq:stokes-complex-reduced-alfeld-bcs} may be seen as 
the Alfeld-split macroelement Stokes complex analog of the lowest-order 
discretization of the de Rham complex with complete polynomials.

Of course, one can further reduce the local complex 
\cref{eq:de-rham-complex-complete-poly-element} to 
the local Whitney complex or lowest-order trimmed polynomial complex:
\begin{equation}
	\label{eq:de-rham-complex-reduce-poly-element}
	\begin{tikzcd}
		\mathbb{R} \arrow[r, "\subset"] 
		& W^{0, h}(K) \arrow[r, "\dee^0"] 
		& W^{1, h}(K) \arrow[r, "\dee^1"] 
		& W^{2, h}(K) \arrow[r, "\dee^2"]
		& W^{3, h}(K) \arrow[r]
		& 0,
	\end{tikzcd}
\end{equation}
where
\begin{subequations}
	\label{eq:whitney-forms-local}
	\begin{alignat}{2}
		W^{0, h}(K) &:= \mathcal{P}_1(K), \qquad &
		W^{1, h}(K) &:= \mathcal{P}_{0}(K)^3 + x \times \mathcal{P}_0(K)^3, \\
		W^{2, h}(K) &:= \mathcal{P}_{0}(K)^3 + x \mathcal{P}_0(K), \qquad &
		W^{3, h}(K) &:= \mathcal{P}_0(K).
	\end{alignat}
\end{subequations}
The following result shows that the two spaces $\tilde{V}^{k, h}(K_A)$, 
$k \in 0:1$, can be reduced further while ensuring that 
\cref{eq:de-rham-complex-reduce-poly-element} is a subcomplex of the 
corresponding local complex.
\begin{theorem}
	\label{thm:cohomology-reduced-alfeld-reduced-bcs}
	Let $\tilde{V}^{k, h}(K_A)$ and $\tilde{V}_{\Gamma_0}^{k ,h}$ be defined as 
	in \cref{thm:cohomology-reduced-alfeld-bcs}. Then, the further reduced 
	spaces 
	\begin{subequations}
		\begin{align}
			\hat{V}^{0, h}(K_A) &:= \{ v \in \tilde{V}^{0, h}(K_A) : 
			\hess v(z) = 0 
			\ \forall z \in \Delta_0(K) \}, \\
			\hat{V}^{1, h}(K_A) &:= \{ v \in \tilde{V}^{1, h}(K_A) : 
			\sym\grad v(z) = 0 
			\ \forall z \in \Delta_0(K)  \},
		\end{align}
	\end{subequations}
	where $\hess$ denotes the Hessian operator,
	satisfy $W^{k, h}(K) \subset \hat{V}^{k, h}(K_A)$, $k \in 0:1$,
	and the complex 
	\begin{equation}
		\label{eq:stokes-complex-reduced-alfeld-reduced-element}
		\begin{tikzcd}
			\mathbb{R} \arrow[r, "\subset"] 
			& \hat{V}^{0, h}(K_A) \arrow[r, "\dee^0"] 
			& \hat{V}^{1, h}(K_A) \arrow[r, "\dee^1"] 
			& \tilde{V}^{2, h}(K_A) \arrow[r, "\dee^2"]
			& \tilde{V}^{3, h}(K_A) \arrow[r]
			& 0
		\end{tikzcd}
	\end{equation}
	is exact. Moreover, if we define for $k \in 0:1$
	\begin{align}
		\label{eq:min-h2h1curl-space-reduce}
		\hat{V}^{k, h} := \{ v \in \tilde{V}^{k, h} : 
		v|_{K} \in \hat{V}^{k, h}(K_A) 
		\ \forall K \in \mesh \} 
		\quad \text{and} \quad 
		\hat{V}^{k, h}_{\Gamma_0} := \hat{V}^{k, h} \cap V^k_{\Gamma_0},
	\end{align}
	then the cohomology of 
	\begin{equation}
		\label{eq:stokes-complex-reduced-alfeld-reduced}
		\begin{tikzcd}
			0 \arrow[r] 
			& \hat{V}_{\Gamma_0}^{0, h} \arrow[r, "\dee^0"] 
			& \hat{V}_{\Gamma_0}^{1, h} \arrow[r, "\dee^1"] 
			& \tilde{V}_{\Gamma_0}^{2, h} \arrow[r, "\dee^2"]
			& \tilde{V}_{\Gamma_0}^{3, h} \arrow[r]
			& 0
		\end{tikzcd}
	\end{equation}
	is isomorphic to the cohomology of \cref{eq:stokes-complex-bcs-feec}.
\end{theorem}
On noting that $\hess W^{0, h}(K) = 0$ and $\sym \grad W^{1, h}(K) = 0$,
the proof of \cref{thm:cohomology-reduced-alfeld-reduced-bcs} is completely
analogous to the proof of \cref{thm:cohomology-reduced-alfeld-bcs}
and is therefore omitted. One may readily see that if any additional degrees 
of freedom from \cref{eq:walkington-dofs} or \cref{eq:h1curl-dofs}
are set to zero, then we would lose the inclusion 
$W^{k, h}(K) \subset \hat{V}^{k, h}(K_A)$, and so 
\cref{eq:stokes-complex-reduced-alfeld-reduced} may be seen as the
Alfeld-split macroelement Stokes complex analog of the Whitney complex.

\subsection{Bounded commuting cochain projections for ``minimal'' complex}

With the cohomology of \cref{eq:stokes-complex-reduced-alfeld-bcs}
fully characterized, the last remaining components used extensively in the 
FEEC literature are bounded commuting cochain projections. One possible
avenue is to modify the construction of locally $L^2$-bounded cochain 
projections in \cite{HuLiangLin25} to take into account the boundary conditions
analogously to the construction of Cl\'{e}ment interpolants \cite{Clement75}.
Instead, we construct Scott-Zhang \cite{ScottZhang90} type interpolants
that also commute. 

To describe the result, we define for an open set $\mathcal{O}$
the following norms: $\| \cdot \|_{V^0(\mathcal{O})}$ the
$H^2(\mathcal{O})$ norm,  $\| \cdot \|_{V^1(\mathcal{O})}$ the 
$H^1(\curl; \mathcal{O})$ norm, $\| \cdot \|_{V^2(\mathcal{O})}$
the $H^1(\mathcal{O})^3$ norm, and $\| \cdot \|_{V^3(\mathcal{O})}$
the $L^2(\mathcal{O})$ norm. Additionally, given a tetrahedron $K \in \mesh$, 
let $\omega_K$ denote the 1 element neighborhood of $K$:
\begin{align*}
	\omega_K := \mathrm{int} \left( \bigcup 
	\{ \bar{K}' \in \mesh : \bar{K} \cap \bar{K}' \neq \emptyset \} \right).
\end{align*}
The locally bounded cochain projections are summarized in the following result.
\begin{theorem}
	\label{thm:bounded-cochain-projections}
	Let $\tilde{\Pi}^3 := V^3 \to \tilde{V}^{3, h}$ be the 
	$L^2(\Omega)$-orthogonal projection. Then, there exist linear projection 
	operators $\tilde{\Pi}^k : V^k \to \tilde{V}^{k, h}$, $k \in 0:2$,
	such that $\{ \tilde{\Pi}^k \}_{k=0}^{3}$ satisfying the following:	
	\begin{enumerate}
		\item[(i)] Trace preservation: 
		$\tilde{\Pi}^k : V^k_{\Gamma_0} \to \tilde{V}_{\Gamma_0}^{k, h}$.
		
		\item[(ii)] Local boundedness: 
		\begin{align}
			\label{eq:cochain-projections-locally-bounded}
			\| \tilde{\Pi}^k v\|_{V^k(K)} \leq C \|v\|_{V^k(\omega_K)}
			\qquad \forall K \in \mesh,
		\end{align}
		where $C$ depends only on $\Omega$, $\Gamma_0$, and shape regularity.
		
		\item[(iii)] Commuting diagram: 
		\begin{equation}
			\label{eq:stokes-complex-reduced-commuting-diagram-bcs}
			\begin{tikzcd}[ampersand replacement=\&]
				V^{0}_{\Gamma_0} \arrow[r, "\grad"] \arrow[d, "\tilde{\Pi}^0"]
				\& V^{1}_{\Gamma_0} \arrow[r, "\curl"] \arrow[d, "\tilde{\Pi}^1"]
				\& V^{2}_{\Gamma_0} \arrow[r, "\div"] \arrow[d, "\tilde{\Pi}^2"]
				\& V^{3}_{\Gamma_0} \arrow[d, "\tilde{\Pi}^3"] \\
				\tilde{V}^{0, h}_{\Gamma_0} \arrow[r, "\grad"] 
				\& \tilde{V}^{1, h}_{\Gamma_0} \arrow[r, "\curl"]
				\& \tilde{V}^{2, h}_{\Gamma_0} \arrow[r, "\div"]
				\& \tilde{V}^{3, h}_{\Gamma_0}.
			\end{tikzcd}
		\end{equation}
	\end{enumerate}
\end{theorem}
The proof of \cref{thm:bounded-cochain-projections} appears below in 
\cref{sec:proof-bounded-cochain-projections}. Owing to the regularity of the 
spaces appearing in the Stokes complex \cref{eq:stokes-complex-bcs-feec}, 
the construction of these interpolants involves standard techniques 
similar to those in \cite{GiraultScott02,ScottZhang90} in contrast to 
the more sophisticated techniques employed for discretizations 
of the de Rham complex \cref{eq:de-rham-complex-bcs-feec}; 
see e.g. \cite{ChaumontVohralik24,ErnGudiSmearsVohralik22,FalkWinther14}
and references therein.

\begin{remark}
	Since the cohomologies of the first two rows of 
	\cref{eq:stokes-complex-reduced-commuting-diagram-bcs} are isomorphic,
	the first two sentences of the proof of \cite[Theorem 5.1]{Arnold18} 
	shows that $\tilde{\Pi}^k$ is an isomorphism between the cohomologies.
	Note that we may replace the first row in
	\cref{eq:stokes-complex-reduced-commuting-diagram-bcs} with
	the spaces $\{ V_{\Gamma_0}^{k,h} \}$ (for any $p \geq 3$), 
	$\{ \tilde{V}_{\Gamma_0}^{k, h} \}$, or 
	$\{ \hat{V}_{\Gamma_0}^{k,h} \}$ (where $\hat{V}_{\Gamma_0}^{k,h} := \tilde{V}_{\Gamma_0}^{k,h}$ for 
	$k \in 2:3$) defined on any other conforming mesh and obtain 
	a commuting diagram. For any of these replacements, the cohomologies 
	of the two complexes in 
	\cref{eq:stokes-complex-reduced-commuting-diagram-bcs} are  
	isomorphic with $\tilde{\Pi}^k$ again being an isomorphism between the 
	cohomologies.
\end{remark}

\begin{remark}
	The operators $\tilde{\Pi}^k$, $k \in 0:1$, may be trivially modified 
	so that $\tilde{\Pi}^k : V^k \to \hat{V}^{k, h}$ and 
	the conclusions of \cref{thm:bounded-cochain-projections} hold with 
	$\tilde{V}^k$ replaced by $\hat{V}^k$.
\end{remark}

\subsection{Outline}

The remainder of the manuscript is organized as follows. In 
\cref{sec:cohomology-full-poly-space}, we show how the framework 
from \cite{HuLiangLin25} applies to the complex with full polynomial
spaces \cref{eq:stokes-complex-full-alfeld-bcs} and modify the framework 
to account for boundary conditions. Then, in \cref{sec:gn-complex-construction},
we construct the reduced spaces $\tilde{V}^{k, h}$, $k \in 0:1$, show that 
the degrees of freedom in \cref{eq:walkington-dofs,eq:h1curl-dofs} are 
unisolvent, and demonstrate how boundary conditions may be incorporated into 
the spaces.
The cohomology of the complex 
\cref{eq:stokes-complex-reduced-alfeld-bcs} is the focus of 
\cref{sec:cohomology-reduced}, and the bounded cochain projections 
are constructed in \cref{sec:proof-bounded-cochain-projections}.

\section{Cohomology of full polynomial spaces}
\label{sec:cohomology-full-poly-space}

We mostly follow the framework in \cite{HuLiangLin25}, originally developed 
for determining the cohomology of discrete complexes spaces without boundary 
conditions, with some modification to handle the case of mixed boundary 
conditions in the Stokes complex. We first require some additional 
notation. Give a collection of $d$-dimensional simplices $\mathcal{S}$, let 
$\Delta(\mathcal{S}) := \bigcup_{\ell=0}^{d} \Delta_{\ell}(\mathcal{S})$
denote the collection of all subsimplices of $\mathcal{S}$. Moreover,
given a simplex $\tau$, we say $\eta \subsimplex \tau$ if 
$\eta$ is a subsimplex of $\tau$ and $\eta \strictsubsimplex \tau$
if $\eta \subsimplex \tau$ and $\eta \neq \tau$.

\subsection{Trace structure}

We begin by defining ambient spaces for the various trace operators we 
will define. 
For $K \in \mesh$, we take $A^k(K) := V^{k, h}(K_A)$.
For $\tau \in \Delta_{1}(\mesh) \cup \Delta_{2}(\mesh)$
and $z \in \Delta_0(\mesh)$, we define 
\begin{alignat*}{2}
	A^0(\tau) &:= \mathcal{P}_{p+2}(\tau) \oplus \mathcal{P}_{p+1}(\tau)^3
	\oplus \bigoplus_{z \in \Delta_0(\tau)} 
	\mathbb{R}^{3 \times 3}_{\sym}, 
	\qquad &
	A^0(z) &:= \mathbb{R} \oplus \mathbb{R}^3 \oplus \mathbb{R}^{3 \times 3}, \\
	A^1(\tau) &:= \mathcal{P}_{p+1}(\tau)^3 \oplus \mathcal{P}_{p}(\tau)^3
	\oplus \bigoplus_{z \in \Delta_0(\tau)} 
	\mathbb{R}^{3 \times 3}, 
	\qquad & 
	A^1(z) &:= \mathbb{R}^3 \oplus \mathbb{R}^{3 \times 3}, \\
	A^2(\tau) &:= \mathcal{P}_{p}(\tau)^3, 
	\qquad &
	A^2(z) &:= \mathbb{R}^3,
\end{alignat*}
where $\mathbb{R}^{3 \times 3}_{\sym}$ denotes the set of $3\times 3$ symmetric 
matrices with real entries. 

For each $k \in 0:3$, we  
show that $A^k := \{ A^k(\tau) : \tau \in \Delta(\mesh) \}$
may be equipped with a trace structure in the sense of 
\cite[Definition 2.1]{HuLiangLin25}.
That is, for $\tau \in \Delta(\mesh)$ and $\eta \subsimplex \tau$,
we define linear operators 
$\Tr^k_{\eta \from \tau} : A^k(\tau) \to A^k(\eta)$ satisfying the 
following properties:
\begin{enumerate}
	\item[(i)] $\Tr^k_{\tau \from \tau}$ is the identity map for all 
	$\tau \in \Delta(\mesh)$.
	
	\item[(ii)] For all $\sigma \in \mesh$ and 
	$\eta \subsimplex \tau \subsimplex \sigma$, there holds 
	\begin{align}
		\label{eq:trace-vanishing-composition}
		\Tr^k_{\eta \from \sigma} u = 0 \implies 
		\Tr^k_{\eta \from \tau} \circ \Tr^k_{\tau \from \sigma} u = 0
		\qquad \forall u \in A^k(\sigma).
	\end{align}
	
	\item[(iii)] $\Tr^k := \{ \Tr^k_{\eta \from \tau} 
	: \eta \subsimplex \tau, \ \tau \in \Delta(\mesh) \}$
	characterizes $V^{k, h}$:
	\begin{multline}
		\label{eq:trace-characterizes-space}
		V^{k, h} = \{ u \in L^2(\Omega) \otimes \mathbb{X}^k : 
		u|_{K} \in A^k(K) \text{ and } \\
		\Tr^k_{\tau \from K} u|_{K} = \Tr^k_{\tau \from K'} u|_{K'} 
		\ \forall \tau \subsimplex K, K', 
		\ \forall K, K' \in \mesh\}.		
	\end{multline}
\end{enumerate}
Recall that $\mathbb{X}^k$ are the global ambient spaces 
defined in \cref{eq:global-ambient-spaces}.
If $(A^k, \Tr^k)$ satisfy (i-ii), then $(A^k, \Tr^k)$ is a trace structure 
in the sense of \cite[Definition 2.1]{HuLiangLin25}, while (iii)
ensures that $V^{k, h}$ is the ``global space'' with respect to 
the trace structure \cite[Definition 2.2]{HuLiangLin25}. In the following 
subsections, we construct the trace operators $\Tr^k_{\eta \from \tau}$ for 
$\eta \neq \tau$, tacitly assuming that (i) always holds.

\subsubsection{Trace structure on $A^0$}

For $K \in \mesh$ and $\tau \strictsubsimplex K$, we define for $v \in A^0(K)$
\begin{align*}
	\Tr^0_{\tau \from K} v &= v|_{\tau} \oplus \grad v|_{\tau} 
	\oplus \bigoplus_{z \in \Delta_0(\tau)} \hess v(z),
\end{align*}
where $w|_{\tau} := w(\tau)$ if $\dim \tau = 0$ and we recall that 
$\hess$ denotes the Hessian operator.
Given $\tau \in \Delta_{1}(\mesh) \cup \Delta_{2}(\mesh)$ 
and $\eta \strictsubsimplex \tau$, we define for $\phi \oplus \psi 
\oplus \bigoplus_{z \in \Delta_0(\tau)} M_z \in A^0(\tau)$
\begin{align*}
	\Tr^0_{\eta \from \tau} \left( \phi \oplus \psi
	\oplus \bigoplus_{z \in \Delta_0(\tau)} M_z \right)
	&= \phi|_{\eta} \oplus \psi|_{\eta}  
	\oplus \bigoplus_{z \in \Delta_0(\eta)} M_z.
\end{align*}
Then, one may readily verify that 
\cref{eq:trace-characterizes-space,eq:trace-vanishing-composition} hold
for $k = 0$.

\subsubsection{Trace structure on $A^1$}

For $K \in \mesh$ and $\tau \strictsubsimplex K$, we define for $v \in A^1(K)$
\begin{align*}
	\Tr^1_{\tau \from K} := \begin{dcases}
		v|_{\tau} \oplus \curl v|_{\tau} 
		\oplus \bigoplus_{z \in \Delta_0(\tau)} \grad v(z)
		& \text{if } \dim \tau > 0, \\
		v(\tau) \oplus \grad v(\tau) & \text{if } \dim \tau = 0.
	\end{dcases}
\end{align*}
For $\tau \in \Delta_1(\mesh) \cup \Delta_2(\mesh)$ and 
$\eta \strictsubsimplex \tau$,
we define for $\phi \oplus \psi \oplus 
\bigoplus_{z \in \Delta_0(\tau)} M_z \in A^1(\tau)$
\begin{align*}
	\Tr^1_{\eta \from \tau} \left( 
	\phi \oplus \psi \oplus 
	\bigoplus_{z \in \Delta_0(\tau)} M_z \right) 
	:= \begin{dcases}
		\phi|_{\eta} \oplus \psi|_{\eta} 
		\oplus \bigoplus_{z \in \Delta_0(\eta)} M_z 
		& \text{if } \dim \eta > 0, \\
		\phi(\eta) \oplus M_{\eta} & \text{if } \dim \eta = 0.
	\end{dcases}
\end{align*}
Then, we may readily verify that 
\cref{eq:trace-vanishing-composition,eq:trace-characterizes-space} hold
for $k = 1$.

\subsubsection{Trace structure on $A^2$ and $A^3$}

For $\tau \in \Delta(\mesh)$ and $\eta \strictsubsimplex \tau$, we define 
for ${v \in A^2(\tau)}$
\begin{align*}
	\Tr_{\eta \from \tau}^2 v = v|_{\eta},
\end{align*}
while for $w \in A^3(\tau)$, we set 
$\Tr_{\eta \from \tau}^3 w = 0$. Then, 
\cref{eq:trace-vanishing-composition,eq:trace-characterizes-space} hold
for $k = 2,3$.

\subsection{Bubble spaces}

Following \cite[Definition 2.3]{HuLiangLin25}, for $K \in \mesh$
and $\tau \subsimplex  K$, we define $B^k(\tau; K)$ by
\begin{align}
	\label{eq:generic-bubble-space}
	B^k(\tau; K) := \{ v \in \Tr^k_{\tau \from K} A^k(K) : 
	\Tr^k_{\eta \from \tau} v = 0 \ \forall \eta \strictsubsimplex \tau \}.
\end{align}
We now show that each trace structure $(A^k, \Tr^k)$ satisfies the 
geometric decomposition property \cite[Definition 2.4]{HuLiangLin25}:
\begin{align}
	\label{eq:geometric-decomposition}
	\sum_{\tau \subsimplex K} \dim B^k(\tau; K) = \dim A^k(K) 
	\qquad \forall K \in \mesh.
\end{align}
In the following subsections, let $K \in \mesh$, $f \in \Delta_2(K)$, 
$e \in \Delta_1(K)$, and $z \in \Delta_0(K)$.

\subsubsection{The case $k=0$}

Expanding definitions, we see that
\begin{align*}
	B^0(K; K) &= V_0^{0, h}(K_A), \\
	B^0(f; K) &= \left\{ v \oplus w 
	\oplus 0 \in A^0(f) : 
	v \in H^2_0(f), \ w \in H^1_0(f)^3, 
	\right. 
	\\
	&\qquad \qquad 
	\left. 
	\text{and }
	(I - \unitvec{n}_f \otimes \unitvec{n}_f) w 
	= \grad_f v 
	\right\}, \\
	B^0(e; K) &= \left\{ v \oplus w 
	\oplus 0 \in A^0(e) : 
	v \in H^3_0(e), w \in H^2_0(e)^3, \text{ and }
	w \cdot \unitvec{t}_e = \partial_{\unitvec{t}_e} v 
	\right\}, \\
	B^0(z; K) &= \mathbb{R} \oplus \mathbb{R}^3 
	\oplus \mathbb{R}_{\sym}^{3 \times 3},
\end{align*}
where $\grad_f$ denotes the surface gradient 
(viewed as an element of $\mathbb{R}^3$).
As a consequence, we obtain
\begin{align*}
	\dim B^0(f; K) &= \dim \mathcal{P}_{p-4}(f) + \dim \mathcal{P}_{p-2}(f), \\
	\dim B^0(e; K) &= \dim \mathcal{P}_{p-4}(e) + \dim \mathcal{P}_{p-3}(e)^2,
\end{align*}
and so performing a direct calculation 
and applying \cite[Lemma 4.8]{FuGuzmanNeilan20}
shows that \cref{eq:geometric-decomposition} holds for $k=0$.

\subsubsection{The case $k=1$}

Expanding definitions gives
\begin{align*}
	B^1(K; K) &=  V_0^{1, h}(K_A), \\
	B^1(f; K) &= \left\{ v \oplus w \oplus 0 \in A^1(f) 
	: v,w \in H^1_0(f)^3, \
	\text{and }
	w \cdot \unitvec{n}_f = \rot_f v
	\right\}, \\
	B^1(e; K) &= \left\{ 
	v \oplus w \oplus 
	0 
	\in A^1(e) 
	: v \in H^2_0(e) \text{ and } w \in H^1_0(e) \right\}, \\
	B^1(z; K) &= \mathbb{R}^3 \oplus \mathbb{R}^{3 \times 3},
\end{align*}
where $\rot_f$ is the surface curl defined so that 
$\rot_f v|_f := \curl v \cdot \unitvec{n}_f|_{f}$ for
$v \in C^{\infty}(\mathbb{R}^3)^3$.
Note that every $v \in \mathcal{P}_{p+1}(f)$ satisfies
\begin{align*}
	v|_{\partial f} = 0
	\text{ and } \rot_f v|_{\partial f} = 0 
	\implies
	\frac{\partial }{\partial \unitvec{t}_e \times \unitvec{n}_f} 
	v \cdot \unitvec{t}_{e} = 0 \qquad \forall e \in \Delta_1(f), 
\end{align*}
and so 
\begin{align*}
	\dim B^1(f; K) &= \dim \mathcal{P}_{p-2}(f)^3 
	-  \sum_{e' \in \Delta_1(f)} \dim \mathcal{P}_{p-2}(e') 
	+ \dim \mathcal{P}_{p-3}(f)^2  \\
	&= \dim \mathcal{P}_{p-2}(f) + \left(  p(p-1) - 
	3(p-1) \right) + \dim \mathcal{P}_{p-3}(f)^2.
\end{align*}
Moreover, $\dim B^1(e; K) = \dim \mathcal{P}_{p-3}(e)^3 
+ \dim \mathcal{P}_{p-2}(e)^3$,
and so performing a direct
calculation and applying \cite[Lemma 4.11]{FuGuzmanNeilan20} 
shows that \cref{eq:geometric-decomposition} holds for $k=1$.

\subsubsection{The case $k=2,3$}

The case $k=2$ corresponds to the usual bubble spaces:
$B^2(K; K) = V^{2, h}_0(K_A)$,
$B^2(\tau; K) = \mathcal{P}_{p}(\tau)^3 \cap H^1_0(\tau)^3$
for $\tau \in \Delta_1(K) \cup \Delta_2(K)$,
and $B^2(z; K) = \mathbb{R}^3$. Moreover, 
$B^3(K; K) = A^3(K)$, and so \cref{eq:geometric-decomposition}
holds for $k=2,3$.

\subsection{Trace complexes}

Again let $K \in \mesh$, $f \in \Delta_2(K)$, 
$e \in \Delta_1(K)$, and $z \in \Delta_0(K)$, and 
consider the following diagram:
\begin{equation}
	\label{eq:full-spaces-trace-complex}
	\begin{tikzcd}[ampersand replacement = \&, column sep = 8em]
		A^0(K) \arrow[r, "\dee_K^0"]
		\arrow[d, "\Tr"]
		\arrow[dd, bend right=60, "\Tr" description]
		\arrow[ddd, bend right=70, "\Tr" description]
		\& A^1(K) \arrow[r, "\dee_K^1"]
		\arrow[d, "\Tr"] 
		\& A^2(K) 
		\arrow[d, "\Tr"] \\
		A^0(f) \arrow[r, "\dee^{0}_{f}"]
		\& A^1(f) \arrow{r}{\dee^{1}_{f}} 
		\& A^2(f) \\
		A^0(e) \arrow{r}{\dee^{0}_{e}}
		\& A^1(e) \arrow{r}{\dee^{1}_{e}} 
		\arrow[from=uu, bend right=60, "\Tr" description, 
		pos=0.66]
		\& A^2(e)
		\arrow[from=uu, bend right=60, "\Tr" description,
		pos=0.66] \\
		A^0(z) \arrow{r}{\dee^{0}_{z}}
		\& A^1(z) \arrow{r}{\dee^{1}_{z}}
		\arrow[from=uuu, bend right=70, "\Tr" description]
		\& A^2(z),
		\arrow[from=uuu, bend right=70, "\Tr" description]
	\end{tikzcd}
\end{equation}
where the vertical arrows are the corresponding trace operators 
(with sub and superscripts omitted) and 
the ``differential'' operators on the horizontal arrows are defined 
as follows: Let $\dee^0_K := \grad$ and $\dee^1_K = \curl$.
For $\tau \in \Delta_1(K) \cup \Delta_2(K)$, let
\begin{align*}
	\dee_{\tau}^0 \left( v \oplus w 
	\oplus \bigoplus_{z \in \Delta_0(\tau)} M_z \right)
	&:= w \oplus 0 
	\oplus \bigoplus_{z \in \Delta_0(\tau)} M_z, \\
	\dee_{\tau}^1 \left( v \oplus w 
	\oplus \bigoplus_{z \in \Delta_0(\tau)} M_{\tau} \right)
	&:= w.
\end{align*}
while for $z \in \Delta_0(K)$, we define 
\begin{align*}
	\dee_z^0 \left( c \oplus v \oplus M \right) &= 
	v \oplus M
	\quad \text{and} \quad 
	\dee_z^1 \left( v \oplus M \right) 
	= \sum_{i, j, k = 1}^{3}\epsilon_{ijk} M_{j, i} \unitvec{e}_k,
\end{align*}
where $\unitvec{e}_k$ is the standard unit vector in the $k$-th direction, and $\epsilon_{ijk}$ is the permutation symbol. Then, one may verify that \cref{eq:full-spaces-trace-complex} commutes
and each row is a complex. Thus, $(A^k, \Tr^k, \dee^k)$  is a conforming 
finite element subcomplex with trace structure (FECTS) 
(\cite[Definition 2.6]{HuLiangLin25}) of the de Rham complex.

\subsection{Compatible currents and bubble complex}

For $k \in 0:3$ and $\tau \in \Delta_k(\mesh)$, 
we recall the  ``currents"  
$\mathcal{I}_{\tau}^k : C(\bar{\Omega}) \otimes \mathbb{X}^k \to \mathbb{R}$
defined in \cref{eq:whitney-dofs-currents}.
In particular, the Stokes formula holds for 
$\sigma \in \Delta_{k+1}(\mesh)$ and 
$w \in C(\bar{\Omega}) \otimes \mathbb{X}^k$:
\begin{align}
	\label{eq:currents-stokes-thm}
	\mathcal{I}^{k+1}_{\sigma}(\dee^k w) 
	= \sum_{\tau \in \Delta_{k}(\sigma)} \mathcal{O}(\tau, \sigma) 
	\mathcal{I}^{k}_{\tau}(w),
\end{align}
where $\mathcal{O}(\tau, \sigma)$ denotes the orientation of $\tau$ relative 
to $\sigma$. Thus, $(\mathcal{I}, \mathbb{R})$ is a family of generalized 
currents \cite[Definition 2.5]{HuLiangLin25}. Note that \cite{HuLiangLin25}
assumes that the domain of $\mathcal{I}_{\tau}^k$ is
$C^{\infty}(\bar{\Omega}) \otimes \mathbb{X}^k$; however, one only needs that 
$\mathcal{I}^k_{\tau}$ is well-defined on $V^{k, h}$.

We also see that for each $k \in 0:3$ and 
$\tau \in \Delta_k(\mesh)$, the functionals
${\tilde{\mathcal{I}}_{\tau}^k : A^k(\tau) \to \mathbb{R}}$ defined by 
\begin{alignat*}{2}
	\tilde{\mathcal{I}}^0_{\tau} ( c \oplus v \oplus M ) &:= c, 
	\qquad & \tilde{\mathcal{I}}^2_{\tau} (w) &:= 
	\int_{\tau} w \cdot \unitvec{n}_{\tau} \d{s}, \\
	\tilde{\mathcal{I}}^1_{\tau} \left(\phi \oplus \psi \oplus
	\bigoplus_{z \in \Delta_0(\tau)} M_z \right) &:= 
	\int_{\tau} \phi \cdot \unitvec{t}_{\tau} \d{s}, 
	& \qquad
	\tilde{\mathcal{I}}^3_{\tau} (q) &:= \int_{\tau} q \d{x},
\end{alignat*}
satisfy
\begin{align}
	\label{eq:current-trace-compatibility}
	\tilde{\mathcal{I}}_{\tau}^k( \Tr^k_{\tau \from \sigma} w) = 
	\mathcal{I}_{\tau}^k(w) \qquad \forall w \in A^{k}(\sigma),
	\ \forall \sigma \in \mesh : \tau \subsimplex \sigma.
\end{align}
Note that above, we are viewing $A^k(\sigma)$ as defined on all of $\Omega$,
where any $V^{k}$-continuous extension is taken.

The bubble spaces \cref{eq:generic-bubble-space} do not depend on the parent 
tetrahedron, so we drop ``$K$" from the notation. Define the modified 
bubble spaces 
\begin{align*}
	\tilde{B}^k(\tau) := \begin{cases}
		B^k(\tau) \cap \ker \tilde{\mathcal{I}}_{\tau}^k 
		& \text{if } k = \dim \tau, \\
		B^k(\tau) & \text{otherwise}.
	\end{cases}
\end{align*}
The next result shows that these modified bubble spaces form an exact 
sequence.
\begin{lemma}
	\label{lem:current-bubble-exactness}
	The following complex is exact for any $\tau \in \Delta(\mesh)$:
	\begin{equation}
		\label{eq:modified-bubble-complex}
		\begin{tikzcd}[ampersand replacement = \&]
			0 \arrow[r] 
			\& \tilde{B}^0(\tau) \arrow[r, "\dee_{\tau}^0"]
			\& \tilde{B}^1(\tau) \arrow[r, "\dee_{\tau}^1"] 
			\& \tilde{B}^2(\tau) \arrow[r, "\dee_{\tau}^2"] 
			\& \tilde{B}^3(\tau) \arrow[r] 
			\& 0.
		\end{tikzcd}
	\end{equation}
\end{lemma}
\begin{proof}
	That \cref{eq:modified-bubble-complex} is a complex follows from 
	the divergence and Stokes theorems. 
	For $\tau \in \mesh$, the exactness of 
	\cref{eq:modified-bubble-complex} follows from 
	\cite[Theorems 3.1-2.2 and Corollaries 3.4-3.6]{FuGuzmanNeilan20}
	(see also \cite[p.1071]{FuGuzmanNeilan20}). 
	For $\tau \in \Delta(\mesh) \setminus \mesh$, 
	$\tilde{B}^3(\tau) = \{0\}$, and we also have 
	\begin{multline*}
		\dim \tilde{B}^0(\tau) + \dim \tilde{B}^2(\tau) 
		- \dim \tilde{B}^1(\tau) \\
		= \dim B^0(\tau) + \dim B^2(\tau) - \dim B^1(\tau) 
		+ (-1)^{\dim \tau + 1} 
		= 0.
	\end{multline*}
	The exactness of \cref{eq:modified-bubble-complex} now follows from a 
	standard counting argument. 
\end{proof}
\noindent Thanks to \cref{eq:current-trace-compatibility}
and \cref{lem:current-bubble-exactness}, the FECTS
$(A^k, \Tr^k, \dee^k)$ is compatible with respect to the currents 
$(\mathcal{I}^k, \mathbb{R})$ in the sense of 
\cite[Definition 2.7]{HuLiangLin25}.

\subsection{Lifted bubble and skeletal complexes}

For $k \in 0:3$, let $(\cdot,\cdot)_{A^k(\tau)}$ denote the natural 
$L^2(\tau)$ inner product on $A^k(\tau)$. Let 
$\mathbb{Q}^k_{\tau} : A^k(\tau) \to \dee_{\tau}^{k-1} \tilde{B}^{k-1}(\tau)$ 
denote the $(\cdot,\cdot)_{A^k(\tau)}$-orthogonal projection onto 
$\dee_{\tau}^{k-1} \tilde{B}^{k-1}(\tau)$:
\begin{align*}
	(\mathbb{Q}^k_{\tau} u, \dee_{\tau}^{k-1} v)_{A^k(\tau)}
	= (u, \dee_{\tau}^{k-1} v)_{A^k(\tau)} 
	\qquad \forall v \in \tilde{B}^{k-1}(\tau),
	\ \forall u \in A^k(\tau).
\end{align*}
We define the harmonic inner product $A^k(\tau)$ 
\cite[Definition 5.1]{HuLiangLin25} as follows:
\begin{align*}
	\langle u, v \rangle_{A^k(\tau)} 
	:= (\mathbb{Q}^k_{\tau} u, \mathbb{Q}^k_{\tau} v)_{A^k(\tau)}
	+ (\dee_{\tau}^k u, \dee_{\tau}^k v)_{A^{k+1}(\tau)}.
\end{align*}
Note that $\langle \cdot, \cdot \rangle_{A^k(\tau)}$ is an inner product 
on $\tilde{B}^k(\tau)$ owing to the exactness of 
\cref{eq:modified-bubble-complex}. 

For $v \in V^{k, h}$, property \cref{eq:trace-characterizes-space} means 
that for $\tau \in \Delta(\mesh)$, we may define the trace of $v$ on $\tau$
independent of the parent tetrahedron; i.e. 
$\Tr^k_{\tau} v := \Tr^k_{\tau \from K} v$,
where $K \in \mesh$ with $\tau \subsimplex K$ is well-defined independent
of the particular choice of $K$. 
As shown in \cite[Lemma 5.2]{HuLiangLin25} and the remaining discussion on 
\cite[p. 29]{HuLiangLin25}, a unisolvent set of degrees of freedom on 
$V^{k, h}$ are 
\begin{subequations}
	\label{eq:global-dofs-bubble}
	\begin{alignat}{2}
		&\langle \Tr^k_{\tau} v, \phi \rangle_{_{A^k(\tau)}} 
		\qquad 
		& &\forall \phi \in \tilde{B}^k(\tau), 
		\ \forall \tau \in \Delta(\mesh), \\
		&\mathcal{I}_{\tau}^k(v) \qquad 
		& &\forall \tau \in \Delta_k(\mesh).
	\end{alignat}
\end{subequations}
Consequently, we define a lift 
of the modified bubble functions and the skeletal space as follows:
\begin{align*}
	\mathbb{B}^k(\tau) &:= \{ v \in V^{k, h} :  
	\mathcal{I}^k_{\eta}(v) = 0 \ \forall \eta \in \Delta_k(\mesh) 
	\text{ and} \\
	&\qquad \qquad 
	\langle \Tr_{\eta}^k v, w \rangle_{A^k(\eta)} = 0 
	\ \forall w \in \tilde{B}^k(\eta), 
	\ \forall \eta \in \Delta(\mesh) \setminus \{\tau\}
	\} \qquad \forall \tau \in \Delta(\mesh),\\
	\mathcal{S}^k &:= \{ v \in V^{k, h} : 
	\langle \Tr_{\eta}^k v, w \rangle_{A^k(\eta)} = 0 
	\ \forall w \in \tilde{B}^k(\eta), 
	\ \forall \eta \in \Delta(\mesh) \}.
\end{align*}
The next lemma summarizes the support properties of these spaces.

\begin{lemma}
	\label{lem:support-skeleton-extended-bubble}
	Let $\tau \in \Delta(\mesh)$. \\
	\noindent (i) For $v \in \mathbb{B}^k(\tau)$,
	there holds
	\begin{align}
		\label{eq:modified-bubble-support}
		\Tr^k_{\eta} v = 0 \qquad \forall \eta \in 
		\left( \bigcup_{\ell=0}^{\dim \tau} \Delta_{\ell}(\mesh) 
		\cup \bigcup_{\ell=\dim \tau + 1}^{3} 
		\{ \rho \in \Delta_{\ell}(\mesh) : \tau \notsubsimplex \rho \}
		\right) \setminus \{\tau\}.
	\end{align}
	
	\noindent (ii)  For $v \in \mathcal{S}^k$, 
	there holds
	\begin{align}
		\label{eq:skeletal-support}
		\mathcal{I}^k_{\eta}(v) = 0 \qquad \forall \eta \in \Delta_k(\tau)
		\implies \Tr^k_{\tau} v = 0.
	\end{align}
	In particular, $\Tr^k_{\tau} v = 0$ if
	$\tau \in \bigcup_{\ell = 0}^{k-1} \Delta_{\ell}(\mesh)$.
\end{lemma}
\begin{proof}
	For (i), let $\eta$ be as in \cref{eq:modified-bubble-support}. Then, 
	we have $\langle \Tr^{k}_{\rho} v, w \rangle_{A^k(\rho)} = 0$
	for all $w \in \tilde{B}^k(\eta)$ and 
	$\rho \in \Delta(\eta)$ and $\mathcal{I}^k_{\rho}(v) = 0$
	for all $\rho \in \Delta_k(\mesh)$. Clearly, $\Tr^k_{z} v = 0$
	for $z \in \Delta_0(\eta)$ and so $\Tr^k_{e} v \in \tilde{B}^k(e)$
	for all $e \in \Delta_1(\eta)$. The result now follows by an induction 
	argument since the degrees of freedom $\langle \cdot, w \rangle_{A^k(\rho)}$
	are unisolvent on $\tilde{B}^k(\rho)$ \cite[Lemma 5.2]{HuLiangLin25}. 
	(ii) follows from similar arguments. 
\end{proof}

Thanks to \cite[Propositions 5.1 \& 5.2]{HuLiangLin25}, each column of 
the following diagram is a direct sum decomposition, 
each row is a complex, and the final row is exact:
\begin{equation}
	\label{eq:stokes-complex-full-geom-decomp-nobcs}
	\begin{tikzcd}[ampersand replacement=\&]
		0 \arrow[r] 
		\& V^{0, h} \arrow[r, "\grad"] 
		\& V^{1, h} \arrow[r, "\curl"] 
		\& V^{2, h} \arrow[r, "\div"]
		\& V^{3, h} \arrow[r]
		\& 0 \\[-2em]
		\& \rotatebox{90}{=} \& \rotatebox{90}{=} \& \rotatebox{90}{=} \&
		\rotatebox{90}{=} \& \\[-2em]
		0 \arrow[r] 
		\& \mathcal{S}^0 \arrow[r, "\grad"] 
		\& \mathcal{S}^1 \arrow[r, "\curl"] 
		\& \mathcal{S}^2 \arrow[r, "\div"]
		\& \mathcal{S}^3 \arrow[r]
		\& 0 \\[-2em]
		\& \bigoplus\limits_{\tau \in \Delta(\mesh)} 
		\& \bigoplus\limits_{\tau \in \Delta(\mesh)} 
		\& \bigoplus\limits_{\tau \in \Delta(\mesh)} 
		\& \bigoplus\limits_{\tau \in \Delta(\mesh)} 
		\& \\[-2em]
		0 \arrow[r] 
		\& \mathbb{B}^0(\tau) \arrow[r, "\grad"] 
		\& \mathbb{B}^1(\tau) \arrow[r, "\curl"] 
		\& \mathbb{B}^2(\tau) \arrow[r, "\div"]
		\& \mathbb{B}^3(\tau) \arrow[r]
		\& 0,
	\end{tikzcd}
\end{equation}
In particular, the cohomology of the first row of 
\cref{eq:stokes-complex-full-geom-decomp-nobcs} is isomorphic to 
the cohomology of the second row, which in turn is isomorphic to 
the de Rham cohomology (the cohomology of \cref{eq:de-rham-complex-bcs-feec}
with $\Gamma_0 = \emptyset$) \cite[Theorem 5.1]{HuLiangLin25}.

\subsection{Incorporating mixed BCs into spaces}

Recall that $\boundarymesh$ is a conforming mesh of $\Gamma_0$. The 
supersmoothness of the spaces $V_{\Gamma_0}^{0, h}$ and $V_{\Gamma_0}^{1, h}$ at mesh vertices 
introduces some subtleties that are not naturally reflected in the trace 
operators. In particular, denote the mesh vertices laying on the 
``flat'' portion of the $\bar{\Gamma}_{0}$ by
\begin{align}
	\Delta_0^{\flat}(\boundarymesh) := 
	\{ z \in \Delta_0(\boundarymesh) 
	: \text{ all faces meeting at $z$ are coplanar} \}.
\end{align}
Then, the following result shows that not all Hessian degrees 
of freedom for functions in $V_{\Gamma_0}^{0, h}$ or gradient 
degrees of freedom vector fields in $V_{\Gamma_0}^{1, h}$ vanish on $\Delta_0^{\flat}$.

\begin{lemma}
	\label{lem:vertex-trace-bcs}
	For $v \in V_{\Gamma_0}^{0, h}$ and $w \in V_{\Gamma_0}^{1, h}$, there holds 
	\begin{subequations}
		\label{eq:vertex-trace-bcs}
		\begin{align}
			\label{eq:vertex-trace-bcs-h2}
			\Tr_{z}^0 v &= \begin{dcases}
				0 \oplus 0 \oplus 
				\partial_{\unitvec{n}_{\Gamma}}^2 v(z) 
				\unitvec{n}_{\Gamma} \otimes \unitvec{n}_{\Gamma} 
				& \text{if } z \in \Delta_0^{\flat}(\boundarymesh), \\
				0 & \text{if } z \in \Delta_0(\boundarymesh) 
				\setminus \Delta_0^{\flat}(\boundarymesh),
			\end{dcases} \\
			\label{eq:vertex-trace-bcs-h1curl}
			\Tr_{z}^1 w &= \begin{dcases}
				0 \oplus \partial_{\unitvec{n}_{\Gamma}} 
				(w \cdot \unitvec{n}_{\Gamma})(z) 
				\unitvec{n}_{\Gamma} \otimes \unitvec{n}_{\Gamma}
				& \text{if } z \in \Delta_0^{\flat}(\boundarymesh), \\
				0 & \text{if } z \in \Delta_0(\boundarymesh) 
				\setminus \Delta_0^{\flat}(\boundarymesh),
			\end{dcases}
		\end{align}
	\end{subequations}
	where $\unitvec{n}_{\Gamma}$ is the outward unit normal vector on 
	$\partial \Omega$.
\end{lemma} 
\begin{proof}
	Assume first that $z \in \Delta_0^{\flat}(\boundarymesh)$ and let 
	$f \in \boundarymesh$ be such that $z \in \Delta_0(f)$.
	Then, $v|_f = 0$ and $\partial_n v|_f = 0$, and so $v(z) = 0$ 
	and $\grad v(z) = 0$. Moreover, differentiating in the tangent plane 
	of $f$ then shows that $D_f^{\alpha} \grad v(z) = 0$ for all 
	$|\alpha| \geq 0$, where $D_f$ denotes the surface differential.
	In particular, 
	$(I - \unitvec{n}_{f} \otimes \unitvec{n}_{f}) \hess v(z) = 0$. 
	Similarly, $w|_f = 0$ and $\curl w|_{f} = 0$ and so 
	$w(z) = 0$, $\grad_f w(z) = 0$, and $\curl w(z) = 0$. Consequently,
	$(I - \unitvec{n}_{f} \otimes \unitvec{n}_{f}) \grad w(z) = 0$.
	The first case of \cref{eq:vertex-trace-bcs-h2} and 
	\cref{eq:vertex-trace-bcs-h1curl} now follows since 
	$\unitvec{n}_f = \unitvec{n}_{\Gamma}$.
	
	Now suppose that $z \in \Delta_0(\boundarymesh) 
	\setminus \Delta_0^{\flat}(\boundarymesh)$. Then, there exist 
	$f, f' \in \boundarymesh$ not coplanar such that 
	$z \in \Delta_0(f)$ and $z \in \Delta_0(f')$. The same arguments 
	as above show that $v(z) = \grad v(z) = 0$, $w(z) = 0$, and 
	\begin{align*}
		(I - \unitvec{\mu} \otimes \unitvec{\mu}) \hess v(z) =
		(I - \unitvec{\mu} \otimes \unitvec{\mu}) \grad w(z)  &= 0,
		\qquad \forall \mu \in \{ \unitvec{n}_f, \unitvec{n}_{f'} \}. 		
	\end{align*}
	Thus, $\hess v(z) = \grad w(z) = 0$.
\end{proof}

One consequence of \cref{lem:vertex-trace-bcs} is that 
the ``zero trace spaces'' defined as the kernel of $\Tr^k_{f}$
for $f \in \boundarymesh$ may not 
coincide with $V_{\Gamma_0}^{k, h}$ for $k=0, 1$; i.e.,
\begin{align*}
	\{ v \in V^{k, h} : \Tr^k_{f} v = 0 
	\text{ if } f \in \boundarymesh \} \subseteq V^{k, h}_{\Gamma_0},
\end{align*}
but the inclusion will generally be strict for $k=0, 1$ if 
$\Delta_0^{\flat}(\boundarymesh)$ is nonempty. 
To rectify this discrepancy, we modify the vertex bubble functions 
further for $z \in \Delta_0^{\flat}(\boundarymesh)$:
\begin{align*}
	\mathbb{B}^{0, nn}(z) 
	&:= \{ v \in  \mathbb{B}^{0}(z) : 
	v(z) = 0, \ \grad v(z) = 0, \
	(I - \unitvec{n}_{\Gamma} 
	\otimes \unitvec{n}_{\Gamma})\hess v(z) 
	= 0 \}, \\
	\mathbb{B}^{1, nn}(z) 
	&:= \{ v \in \mathbb{B}^{1}(z) : 
	v(z) = 0, \
	(I - \unitvec{n}_{\Gamma} 
	\otimes \unitvec{n}_{\Gamma}) 
	\grad v(z) = 0\}.
\end{align*}
The corresponding modified bubble spaces, which will be shown to satisfy 
the boundary conditions, are given by 
\begin{align*}
	\mathbb{B}_{\Gamma_0}^k(\tau) := \begin{cases}
		\mathbb{B}^k(\tau) & \text{if } \tau \in 
		\Delta(\mesh) 
		\setminus \Delta(\boundarymesh), \\
		\mathbb{B}^{k, nn}(\tau) 
		& \text{if } \tau \in \Delta_0^{\flat}(\boundarymesh) 
		\text{ and } k < 2, \\
		\{0\} & \text{otherwise},
	\end{cases} \qquad \forall \tau \in \Delta(\mesh), \ \forall k \in 0:3.
\end{align*}
The skeletal spaces with boundary conditions are simply
\begin{align*}
	\mathcal{S}_{\Gamma_0}^k := \{ v \in \mathcal{S}^k : \mathcal{I}^k_{\tau}(v) = 0
	\ \forall \tau \in \Delta_k(\boundarymesh) \}
	\qquad \forall k \in 0:3. 
\end{align*}
\noindent The next result shows that these spaces do indeed satisfy the boundary 
conditions.
\begin{lemma}
	For $k \in 0:3$, $\mathcal{S}^k_{\Gamma_0} \subset V_{\Gamma_0}^{k, h}$
	and $\mathbb{B}_{\Gamma_0}^k(\tau) \subset V_{\Gamma_0}^{k, h}$ for all 
	$\tau \in \Delta(\mesh)$.
\end{lemma}
\begin{proof}
	Let $v \in \mathcal{S}^k_{\Gamma_0}$. For any $f \in \boundarymesh$,
	$\mathcal{I}^k_{\eta}(v) = 0$ for all $\eta \in \Delta_k(f)$, and so 
	$\Tr^k_{f} v = 0$ thanks to 
	\cref{lem:support-skeleton-extended-bubble}. 
	Thus, $v \in V_{\Gamma_0}^{k, h}$.
	
	Now let $\tau \in \Delta(\mesh)$ and 
	$v \in \mathbb{B}_{\Gamma_0}^k(\tau)$. 
	If $\tau \notin \Delta(\boundarymesh)$,
	then $\Tr^k_{f} v = 0$ for all $f \in \boundarymesh$
	thanks to \cref{lem:support-skeleton-extended-bubble},
	and so $v \in V_{\Gamma_0}^{k, h}$. Now suppose that 
	$\tau \in \Delta_0^{\flat}(\boundarymesh)$. Applying 
	\cref{lem:support-skeleton-extended-bubble} once again gives 
	$\Tr_{f}^k v = 0$ for $f \in \boundarymesh$
	with $\tau \notsubsimplex f$. Consequently, the final case to verify 
	is $f \in \boundarymesh$ with $\tau \subsimplex f$.
	
	For $k = 0$, expanding the definition of the harmonic inner products
	in the condition $v \in \mathbb{B}_{\Gamma_0}^k(\tau)$ shows that	
	$v_f := v|_{f} \in \mathcal{P}_{p+2}(f)$ satisfies 
	\begin{subequations}
		\label{eq:proof:Pp2-face-dofs}
		\begin{alignat}{2}
			D_f^{\alpha} v_f (z) &= 0 \qquad & &\forall |\alpha| \leq 2, 
			\ \forall z \in \Delta_0(f), \\
			(\partial_{\unitvec{t}_{\partial f}}v_f, \partial_{\unitvec{t}_{\partial f}}w)_{L^2(e)} &= 0 
			\qquad & &\forall w \in \mathcal{P}_{p+2}(e) \cap H^3_0(e), \ 
			\forall e \in \Delta_1(f), \\
			(\partial_{\unitvec{n}_{\partial f}} v_{f}, w )_{L^2(e)} &= 0 
			\qquad & &\forall w \in \mathcal{P}_{p+1}(e) \cap H^2_0(e), \ 
			\forall e \in \Delta_1(f), \\ 
			(\operatorname{grad}_fv_f, \operatorname{grad}_fw)_{L^2(f)} &=0 
			\qquad & &\forall w \in \mathcal{P}_{p+2}(f) \cap H^2_0(f),
		\end{alignat}	
	\end{subequations}
	where we recall $D_f$ denotes the surface differential.
	Since the above degrees of freedom are unisolvent on $\mathcal{P}_{p+2}(f)$,
	$v_{f} \equiv 0$. Similarly, 
	$u_f = \partial_{n_f} v|_f \in \mathcal{P}_{p+1}(f)$ satisfies
	\begin{subequations}
		\label{eq:proof:Pp1-face-dofs}
		\begin{alignat}{2}
			D_f^{\alpha} u_f (z) &= 0 \qquad & &\forall |\alpha| \leq 1, 
			\ \forall z \in \Delta_0(f), \\
			(u_f, w)_{L^2(e)} &= 0 
			\qquad & &\forall w \in \mathcal{P}_{p+1}(e) \cap H^2_0(e), \ 
			\forall e \in \Delta_1(f),  \\
			\label{eq:proof:Pp1-face-dofs-interior-face}
			(u_f, w)_{L^2(f)} &=0 
			\qquad & &\forall w \in \mathcal{P}_{p+1}(f) \cap H^1_0(f),
		\end{alignat}	
	\end{subequations}	
	and so $u_f \equiv 0$. As a result, $v \in V_{\Gamma_0}^{0, h}$.
	
	For $k=1$, each component of 
	$v_f := v|_{f} \in \mathcal{P}_{p+1}(f)^3$ also satisfies 
	\begin{subequations}
		\label{eq:proof:Pp3-face-dofs}
		\begin{alignat}{2}
			D_f^{\alpha} v_f (z) &= 0 \qquad & &\forall |\alpha| \leq 1, 
			\ \forall z \in \Delta_0(f), \\
			(v_f, w)_{L^2(e)} &= 0 
			\qquad & &\forall w \in \mathcal{P}_{p+1}(e)^3 \cap H^2_0(e)^3, \ 
			\forall e \in \Delta_1(f),  \\
			\label{eq:proof:Pp3-face-dofs-interior-face}
			(v_f, \operatorname{grad}_f w)_{L^2(f)} &=0 
			\qquad & &\forall w \in \mathcal{P}_{p+2}(f) \cap H^2_0(f),\\
			(\operatorname{rot}_fv_f, \operatorname{rot}_fw)_{L^2(f)} &=0 
			\qquad & &\forall w \in \mathcal{P}_{p+1}(f)^3 \cap H^1_0(f)^3,
		\end{alignat}	
	\end{subequations}	 and so 
	$v_f \equiv 0$. Moreover, $u_f := \curl v|_f
	\in \mathcal{P}_p(f)^3$ satisfies	
	\begin{subequations}
		\label{eq:proof:Pp-face-dofs}
		\begin{alignat}{2}
			u_f (z) &= 0 
			\qquad & &\forall z \in \Delta_0(f), \\
			(u_f, w)_{L^2(e)} &= 0 
			\qquad & &\forall w \in \mathcal{P}_{p}(e)^3 
			\cap H^1_0(e)^3, \ \forall e \in \Delta_1(f), \\
			(u_f, w)_{L^2(f)} &=0 
			\qquad & &\forall w \in \mathcal{P}_{p}(f)^3 
			\cap H^1_0(f)^3,
		\end{alignat}	
	\end{subequations}
	and so $u_f \equiv 0$. Consequently, $v \in V_{\Gamma_0}^{1, h}$.
\end{proof}

We also have the analog of the diagram 
\cref{eq:stokes-complex-full-geom-decomp-nobcs}:
\begin{lemma}
	Each column of the diagram below is a direct sum decomposition, 
	each row is a complex, and the final row is exact.
	\begin{equation}
		\label{eq:stokes-complex-full-geom-decomp}
		\begin{tikzcd}[ampersand replacement=\&]
			0 \arrow[r] 
			\& \tilde{V}_{\Gamma_0}^{0, h} \arrow[r, "\grad"] 
			\& \tilde{V}_{\Gamma_0}^{1, h} \arrow[r, "\curl"] 
			\& \tilde{V}_{\Gamma_0}^{2, h} \arrow[r, "\div"]
			\& \tilde{V}_{\Gamma_0}^{3, h} \arrow[r]
			\& 0 \\[-2em]
			\& \rotatebox{90}{=} \& \rotatebox{90}{=} \& \rotatebox{90}{=} \&
			\rotatebox{90}{=} \& \\[-2em]
			0 \arrow[r] 
			\& \mathcal{S}_{\Gamma_0}^0 \arrow[r, "\grad"] 
			\& \mathcal{S}_{\Gamma_0}^1 \arrow[r, "\curl"] 
			\& \mathcal{S}_{\Gamma_0}^2 \arrow[r, "\div"]
			\& \mathcal{S}_{\Gamma_0}^3 \arrow[r]
			\& 0 \\[-2em]
			\& \bigoplus\limits_{\tau \in \Delta(\mesh)} 
			\& \bigoplus\limits_{\tau \in \Delta(\mesh)} 
			\& \bigoplus\limits_{\tau \in \Delta(\mesh)} 
			\& \bigoplus\limits_{\tau \in \Delta(\mesh)} 
			\& \\[-2em]
			0 \arrow[r] 
			\& \mathbb{B}_{\Gamma_0}^0(\tau) \arrow[r, "\grad"] 
			\& \mathbb{B}_{\Gamma_0}^1(\tau) \arrow[r, "\curl"] 
			\& \mathbb{B}_{\Gamma_0}^2(\tau) \arrow[r, "\div"]
			\& \mathbb{B}_{\Gamma_0}^3(\tau) \arrow[r]
			\& 0.
		\end{tikzcd}
	\end{equation}
\end{lemma}
\begin{proof}	
	\noindent \textbf{Step 1: Direct sum decomposition. }
	The column for $k=3$ is identical to 
	\cref{eq:stokes-complex-full-geom-decomp-nobcs}, so 
	consider $k \in 0:2$.
	Let $v \in V_{\Gamma_0}^{k, h}$. Thanks 
	to the direct sum decomposition in 
	\cref{eq:stokes-complex-full-geom-decomp-nobcs}, we have 
	\begin{align*}
		v = v_S + \sum_{\tau \in \Delta(\mesh)} v_{\tau},
		\quad \text{with} \quad v_S \in \mathcal{S}^k, 
		\ v_{\tau} \in \mathbb{B}^k(\tau), 
		\ \tau \in \Delta(\mesh).
	\end{align*}
	For $\eta \in \Delta_k(\boundarymesh)$, we have 
	$\mathcal{I}_{\eta}^{k}(v_S) = \mathcal{I}_{\eta}^{k}(v) = 0$,
	and so $v_S \in \mathcal{S}_{\Gamma_0}^k$.
	
	Note that $\mathbb{B}_{\Gamma_0}^k(\tau) = \mathbb{B}^k(\tau)$
	if $\tau \notin \Delta(\boundarymesh)$, so 
	suppose that $\tau \in \Delta(\boundarymesh)$.
	By \cref{lem:support-skeleton-extended-bubble}, we have 
	\begin{align*}
		\langle \Tr_{\tau}^k v_{\tau}, w \rangle_{A^k(\tau)}
		= \langle \Tr_{\tau}^k v, w \rangle_{A^k(\tau)}
		\qquad \forall w \in \tilde{B}^k(\tau).
	\end{align*}
	If $\dim \tau \geq 1$ or $k = 2$, then the above quantity vanishes
	and so $v_{\tau} \equiv 0$ since 
	the above degrees of freedom are unisolvent 
	on $\tilde{B}^k(\tau)$ \cite[Lemma 5.2]{HuLiangLin25}.
	For $\dim \tau = 0$ and $k \in 0:1$, 
	\cref{lem:vertex-trace-bcs} shows that 
	$v_{\tau} \in \mathbb{B}^k_{\Gamma_0}(\tau)$.
	Thus, each column of \cref{eq:stokes-complex-full-geom-decomp}
	is a direct sum decomposition. \\
	
	\noindent \textbf{Step 2: Complex property. }
	The second row of \cref{eq:stokes-complex-full-geom-decomp} is a complex
	since the second row of \cref{eq:stokes-complex-full-geom-decomp-nobcs}
	is a complex and \cref{eq:currents-stokes-thm} holds.
	Thus, it remains to show that 
	the final row is an exact complex. For 
	$\tau \in \Delta(\mesh) \setminus \Delta_0^{\flat}(\boundarymesh)$, 
	we have $\mathbb{B}_{\Gamma_0}^k(\tau) = \mathbb{B}^k(\tau)$ 
	for all $k \in 0:3$ 
	or $\mathbb{B}_{\Gamma_0}^k(\tau) = \{0 \}$ for all $k \in 0:3$, 
	both of which are exact complexes thanks to 
	\cref{eq:stokes-complex-full-geom-decomp-nobcs}.
	
	Now suppose that $\tau \in \Delta_0^{\flat}(\boundarymesh)$, 
	for which the bottom row of \cref{eq:stokes-complex-full-geom-decomp} 
	reads
	\begin{equation*}
		\begin{tikzcd}[ampersand replacement=\&]
			0 \arrow[r] 
			\& \mathbb{B}^{0, nn}(\tau) \arrow[r, "\grad"] 
			\& \mathbb{B}^{1, nn}(\tau) \arrow[r, "\curl"] 
			\& 0 \arrow[r, "\div"]
			\& 0 \arrow[r]
			\& 0.
		\end{tikzcd}
	\end{equation*}
	$\mathbb{B}^{0, nn}(\tau)$ and 
	$\mathbb{B}^{1, nn}(\tau)$ clearly both have dimension 1,
	and so we only need to verify that 
	$\grad \mathbb{B}^{0, nn}(\tau) = 
	\mathbb{B}^{1, nn}(\tau)$. Since the bottom row of 
	\cref{eq:stokes-complex-full-geom-decomp-nobcs} is a complex
	and 
	\begin{align*}
		\Tr_{\tau}^1 \grad v 
		= 0 \oplus \partial_{\unitvec{n}_{\Gamma}}^2
		v(\tau) \unitvec{n}_{\Gamma} 
		\otimes \unitvec{n}_{\Gamma}
		\qquad \forall v \in \mathbb{B}^{0, nn}(\tau),
	\end{align*}
	we have
	$\grad \mathbb{B}^{0, nn}(\tau) = 
	\mathbb{B}^{1, nn}(\tau)$.
\end{proof}

\subsection{Cohomology via relative simplicial cochain complex}
\label{sec:proof-cohomology-full-spaces-bcs}

We now introduce the relative simplicial cochain complex, modifying the 
presentation in \cite{Licht17}.
For $k \in 0:3$, let $C_k(\mathcal{V})$ denote the space of simplicial 
$k$-chains (i.e. formal linear combinations of subsimplices of dimension $k$)
of a mesh $\mathcal{V}$.
Then, $C_k(\boundarymesh)$ is a subspace of $C_k(\mesh)$, and 
we define $C_k(\mesh, \boundarymesh) 
:= C_k(\mesh) / C_k(\boundarymesh)$. Let 
${\partial_k : C_k(\mesh, \boundarymesh) 
\to C_{k-1}(\mesh, \boundarymesh)}$
denote the standard simplicial boundary operator defined on the quotient space
uniquely by the condition
\begin{align*}
	\partial_k (\tau + C_k(\boundarymesh)) 
	= \sum_{\substack{\eta \in \Delta_{k-1}(\tau) \\ 
			\eta \notin \Delta_{k-1}(\boundarymesh)}}
	\mathcal{O}(\eta, \tau) (\eta + C_{k-1}(\boundarymesh))
	\qquad \forall \tau \in \Delta_k(\mesh) 
	\setminus \Delta_k(\boundarymesh).
\end{align*}
Let $\dual{\partial}_k : \dual{C_k(\mesh, \boundarymesh)} 
\to \dual{C_{k+1}(\mesh, \boundarymesh)}$ denote the corresponding 
cochain map defined uniquely by the condition:
\begin{align*}
	(\dual{\partial}_k \ell_k)(\omega) = \ell_k( \partial_{k+1} \omega)
	\qquad \forall \omega \in C_{k+1}(\mesh, \boundarymesh), \
	\forall \ell_k \in \dual{C_k(\mesh, \boundarymesh)}.
\end{align*}

\begin{proof}[Proof of \cref{thm:cohomology-full-alfeld-bcs}]
	Let $k \in 0:3$. 
	A simple consequence of \cite[Lemma 5.2]{HuLiangLin25} is that the currents 
	$\{ \mathcal{I}^k_{\eta} : \eta \in \Delta_k(\mesh) \}$
	are a unisolvent set of degrees of freedom on $\mathcal{S}^k$, and so
	$\{ \mathcal{I}^k_{\eta} : \eta \in \Delta_k(\mesh) 
	\setminus \Delta_k(\boundarymesh) \}$ are unisolvent on 
	$\mathcal{S}^k_{\Gamma_0}$. Thus, $\mathcal{S}^k_{\Gamma_0}$ and 
	$\dual{\mathcal{C}_k(\mesh, \boundarymesh)}$
	have the same dimension. We define 
	$\pi^k : \mathcal{S}^k_{\Gamma_0} \to \dual{\mathcal{C}_K(\mesh, \boundarymesh)}$
	uniquely by the condition 
	\begin{align*}
		(\pi^k v)(\tau + C_{k}(\boundarymesh)) = \mathcal{I}^k_{\tau}(v) 
		\qquad \forall \tau \in \Delta_{k}(\mesh) 
		\setminus \Delta_{k}(\boundarymesh), \
		\forall v \in \mathcal{S}^k_{\Gamma_0}.
	\end{align*}
	By unisolvency, $\pi^k$ is injective and hence surjective. Thanks 
	to \cref{eq:currents-stokes-thm}, we also have 
	the following commutativity for all $v \in \mathcal{S}_{\Gamma_0}^k$ and 
	$\tau \in \Delta_k(\mesh) \setminus \Delta_k(\boundarymesh)$: 
	\begin{align*}
		\dual{\partial}_k (\pi^k v)(\tau + C_{k+1}(\boundarymesh))
		&= \sum_{\substack{\eta \in \Delta_{k-1}(\tau) \\ 
				\eta \notin \Delta_{k-1}(\boundarymesh)}}  
		\mathcal{O}(\eta, \tau) (\pi^k v)(\eta + C_{k}(\boundarymesh))  \\
		&= \sum_{\substack{\eta \in \Delta_{k-1}(\tau) \\ 
				\eta \notin \Delta_{k-1}(\boundarymesh)}}
		\mathcal{O}(\eta, \tau) \mathcal{I}_{\eta}^k(v) \\
		&= \mathcal{I}_{\tau}(\dee^k v) \\
		&= (\pi^{k+1} \dee^k v)(\tau + C_{k+1}(\boundarymesh)),
	\end{align*}
	where we used that 
	$\mathcal{I}_{\eta}^k(v) = \mathcal{I}_{\rho}^{k+1}(\dee^k v) = 0$ 
	for all $\eta \in \Delta_k(\boundarymesh)$ and 
	$\rho \in \Delta_{k+1}(\boundarymesh)$. Thus, the diagram 
	\begin{equation*}
		\begin{tikzcd}[ampersand replacement=\&]
			0 \arrow[r] \arrow[d]
			\& \mathcal{S}_{\Gamma_0}^0 \arrow[r, "\grad"] \arrow[d, "\pi^0"] 
			\& \mathcal{S}_{\Gamma_0}^1 \arrow[r, "\curl"] \arrow[d, "\pi^1"]
			\& \mathcal{S}_{\Gamma_0}^2 \arrow[r, "\div"] \arrow[d, "\pi^2"]
			\& \mathcal{S}_{\Gamma_0}^3 \arrow[r] \arrow[d, "\pi^3"]
			\& 0 \arrow[d] \\
			0 \arrow[r] 
			\& \dual{C^0(\mesh, \boundarymesh)} \arrow[r, "\dual{\partial}_0"] 
			\& \dual{C^1(\mesh, \boundarymesh)} \arrow[r, "\dual{\partial}_1"] 
			\& \dual{C^2(\mesh, \boundarymesh)} \arrow[r, "\dual{\partial}_2"]
			\& \dual{C^3(\mesh, \boundarymesh)} \arrow[r]
			\& 0,
		\end{tikzcd}
	\end{equation*}
	commutes. Since $\pi^k$, $k \in 0:3$ are isomorphisms, the two 
	sequences have isomorphic cohomologies. Additionally, the 
	singular homology group and the relative simplicial homology group 
	have isomorphic cohomologies since $\Omega$ and $\Gamma_0$ admit a
	conforming mesh 
	(see e.g. \cite[Chapter 4, Section 6, Theorem 8]{Spanier95}).
\end{proof}

\section{Construction of a ``minimal'' conforming complex}
\label{sec:gn-complex-construction}

We now turn to the construction of a ``minimal'' conforming 
subcomplex of \cref{eq:stokes-complex-full-alfeld-bcs}. With 
$\tilde{V}^{2, h}(K_A) = \GN(K_A)$ and $\tilde{V}^{3, h}(K_A) = \DG^0(K)$
as in \cref{sec:minimal-complexes-results},
we seek discrete spaces $\tilde{V}^{k, h}(K_A) \subset V^{k, h}(K_A)$, 
$k \in 0:1$, so that 
\cref{eq:stokes-complex-reduced-alfeld-element} is an exact complex.
We want the spaces to be sufficiently large so that the inclusions 
$\mathcal{P}_{3-k}(K) \otimes \mathbb{X}^k \subset \tilde{V}^{k, h}(K_A)$ hold, 
but also minimal in the sense that 
degrees of freedom in \cref{eq:walkington-dofs} for $k=0$ and 
\cref{eq:h1curl-dofs} for $k=1$ are unisolvent. We achieve this by defining 
the spaces $\tilde{V}^{k, h}(K_A)$ as subspaces of $V^{k, h}(K_A)$ satisfying
particular constraints.

\subsection{Local $H^2$-conforming space $\tilde{V}^{0, h}(K_A)$}

Let $K \in \mesh$ and suppose that the bilinear form
${a_K(\cdot,\cdot) : H^1(K)^3 \times H^1(K)^3 \to \mathbb{R}}$ 
is continuous and satisfies:
\begin{subequations}
	\label{eq:kill-h2-interior-bilinear}
	\begin{alignat}{2}
		\label{eq:kill-h2-interior-bilinear-kill-p2}
		a_K(u, \grad \phi) &= 0 
		\qquad & &\forall u \in \mathcal{P}_2(K)^3, 
		\ \forall \phi \in V^{0, h}_0(K_A), \\
		\label{eq:kill-h2-interior-bilinear-coercive}
		a_K(\grad \psi, \grad \psi) &\geq 
		\gamma_a \|\psi\|_{H^2(K)}^2 
		\qquad & &\forall \psi \in V^{0, h}_0(K_A),
	\end{alignat} 
\end{subequations}
where $\gamma_a$ is independent of $K$ and we recall 
that $V^{0, h}_0(K_A) = V^{0, h}(K_A) \cap H^2_0(K)$.
One bilinear form satisfying \cref{eq:kill-h2-interior-bilinear}
is $a_K(\cdot,\cdot) = (\grad \cdot, \grad \cdot)_{L^2(K)}$.
For each face $f \in \Delta_2(K)$, let 
$\ell_f \in \mathcal{P}_3(f)^*$ be any linear functional 
with the following properties:
\begin{align}
	\label{eq:kill-cubic-bubble-dof}
	\ell_f(p) = 0 \qquad \forall p \in \mathcal{P}_2(f)
	\quad \text{and} \quad \ell_f(b_f) \neq 0.
\end{align}
For example, one could take 
\begin{align*}
	\ell_f(p) = \int_f (I - \mathbb{P}_{2, f})b_f p \d{s} 
	\qquad \forall p \in \mathcal{P}_3(f),
\end{align*}
where $\mathbb{P}_{2, f} : L^1(f) \to \mathcal{P}_2(f)$
is the $L^2(f)$-orthogonal projector onto $\mathcal{P}_2(f)$.

Consider the following subspace of $V^{0, h}(K_A)$:
\begin{multline}
	\label{eq:local-h2-space-def}
	\tilde{V}^{0, h}(K_A) := \{ v \in V^{0, h}(K_A) :
	\partial_n v|_{f} \in \mathcal{P}_3(f) 
	\text{ and } \ell_f(\partial_n v|_{f}) = 0  
	\ \forall f \in \Delta_2(K) \\
	\text{and } a_K(\grad v, \grad w) = 0 
	\ \forall w \in V^{0, h}_0(K_A) \}.
\end{multline}
The following result summarizes the key properties of $\tilde{V}^{0, h}(K_A)$.
\begin{lemma}
	\label{lem:walkington-dofs-unisolvent}
	$\dim \tilde{V}^{0, h}(K_A) = 40$, 
	the degrees of freedom in \cref{eq:walkington-dofs} with 
	$\mesh = K_A$ are unisolvent
	on $\tilde{V}^{0, h}(K_A)$,
	and $\mathcal{P}_3(K) \subset \tilde{V}^{0, h}(K_A)$.
\end{lemma}
\begin{proof}
	We follow similar arguments as in the proof of
	\cite[Lemma 2.3]{Walkington14}. 
	Note that the number of independent linear functionals 
	on $V^{0, h}(K_A)$ in 
	\cref{eq:walkington-dofs} is 40. \\
	
	\noindent \textbf{Step 1: $\dim \tilde{V}^{0, h}(K_A) \leq 40$. } 
	Assume that $v \in \tilde{V}^{0, h}(K_A)$ and the degrees of freedom in 
	\cref{eq:walkington-dofs} vanish. Then, as shown in the proof of
	\cite[Lemma 2.3]{Walkington14}, $v|_{\partial K} = 0$ and 
	\begin{align*}
		D_f^{\alpha} (\partial_n v|_{f})(z) = 0 \qquad 
		\forall |\alpha| \leq 1, \ 
		\forall z \in \Delta_0(f), \ \forall f \in \Delta_2(K),
	\end{align*}
	where we recall that $D_f$ denotes the surface differential. Thus, 
	$\partial_n v|_{f} \in \spann\{ b_f \}$ and 
	$\ell_f(\partial_n v|_{f}) = 0$, so $v \in V^{0, h}_0(K)$. 
	The coercivity of $a_K(\cdot,\cdot)$
	\cref{eq:kill-h2-interior-bilinear-coercive}
	gives $v \equiv 0$. Consequently, $\dim \tilde{V}^{0, h}(K_A) \leq 40$. \\
	
	\noindent \textbf{Step 2: $\dim \tilde{V}^{0, h}(K_A) \geq 40$. } 	
	Suppose we are given
	arbitrary values for the degrees of freedom in \cref{eq:walkington-dofs}
	$\{ c_{z}^{\alpha} : |\alpha| \leq 2, \ z \in \Delta_0(K) \}$.
	We now show that there exists $v \in \tilde{V}^{0, h}(K_A)$ with 
	these degrees of freedom by using a minor modification of
	the degrees of freedom for the space $V^{0, h}(K_A)$ in 
	\cite[Lemma 4.8]{FuGuzmanNeilan20}. 
	
	For each $f \in \Delta_2(K)$, let $w_{n, f} \in \mathcal{P}_3(f)$ be
	the unique cubic polynomial satisfying
	\begin{alignat*}{2}
		D_f^{\alpha} w_{n, f}(z) &= \sum_{i=1}^{3} c_{z}^{\alpha + e_i} 
		(\unitvec{n}_f \cdot \unitvec{e}_i)
		\qquad & &\forall |\alpha| \leq 1, 
		\ \forall z \in \Delta_0(f), \\
		\ell_{f} (w_{n, f}) &= 0. \qquad & & 
	\end{alignat*}
	Thanks to \cite[Lemma 4.8]{FuGuzmanNeilan20}, there exists 
	$v \in V^{0, h}(K_A)$ satisfying
	\begin{alignat*}{2}
		D^{\alpha} v(z) &= c_{z}^{\alpha} 
		\qquad & &\forall |\alpha| \leq 2, \ \forall z \in \Delta_0(K), \\
		\int_{e} \grad v \cdot \unitvec{n}_f \d{s} &= \int_{e} w_{n, f} \d{s}
		\qquad & &\forall e \in \Delta_1(f), \ \forall f \in \Delta_2(K), \\
		\int_f (\grad v \cdot \unitvec{n}_f) \kappa \d{s} 
		&= \int_f w_{n, f} \kappa \d{s} 
		\qquad & &\forall \kappa \in \mathcal{P}_1(f), 
		\ \forall f \in \Delta_2(K), \\
		a_K(\grad v, \grad \phi) &= 0
		\qquad & &\forall \phi \in V^{0, h}_0(K_A).  		
	\end{alignat*}
	By construction, the degrees of freedom in \cref{eq:walkington-dofs} of 
	$v$ match $\{ c_{z}^{\alpha}\}$ in the sense that
	\begin{alignat*}{2}
		D^{\alpha} v(z) &= c_{z}^{\alpha} 
		\qquad & &\forall |\alpha| \leq 2, \ \forall z \in \Delta_0(K).
	\end{alignat*}
	We finish the proof by showing that for $f \in \Delta_2(K)$,
	$\partial_{n}  v|_{f} = w_{n, f}$ and thus $v \in \tilde{V}^{0, h}(K_A)$. 
	By construction, we have that 
	$v_f := (\partial_n v|_f - w_{n, f}) \in \mathcal{P}_{4}(f)$
	satisfies \cref{eq:proof:Pp1-face-dofs} with $p=3$, and so 
	$\partial_n v|_{f}  = w_{n, f}$.	\\
	
	\noindent \textbf{Step 3: Inclusion of cubics. }
	Note that for any $v \in \mathcal{P}_3(K)$, 
	\cref{eq:kill-h2-interior-bilinear-kill-p2} gives 
	$a_K(\grad v, \grad w) = 0$ for all $w \in V^{0, h}_0(K_A)$,
	while \cref{eq:kill-cubic-bubble-dof} gives $\ell_f(\partial_n v|_{f}) = 0$
	for all $f \in \Delta_2(K)$,
	and so $\mathcal{P}_3(K) \subset \tilde{V}^{0, h}(K_A)$.
\end{proof}

\begin{remark}
	An alternative reduction of $V^{0, h}(K_A)$ was introduced in 
	\cite{Walkington14}, in which
	the restriction $\partial_n v|_f \in \mathcal{P}_3(f)$ 
	for all $f \in \Delta_2(K)$ was also imposed and
	$C^4$ continuity is imposed at the barycenter of $K$. This supersmooth space
	locally reproduces $\mathcal{P}_{4}(K)$ and the degrees of freedom 
	are \cref{eq:walkington-dofs} augmented with one degree of freedom per face 
	for the normal derivative and a single interior degree of freedom.
\end{remark}

\subsection{Local $H^1(\curl)$-conforming space $\tilde{V}^{1, h}(K_A)$}

Note that $\tilde{V}^{0, h}(K_A)$ chosen as in \cref{eq:local-h2-space-def}
appears to the left of $\tilde{V}^{1, h}(K_A)$ in
the complex \cref{eq:stokes-complex-reduced-alfeld-element}, while
$\GN(K_A)$ appears to the right. To satisfy the complex property,
we simply define $\tilde{V}^{1, h}(K_A)$ to be the subspace of
$V^{1, h}(K_A)$ satisfying the constraints
imposed by these two spaces:
\begin{multline}
	\label{eq:local-h1curl-space-def}
	\tilde{V}^{1, h}(K_A) := \{ 
	v \in V^{1, h}(K_A) : 
	v \cdot \unitvec{n}_f|_{f} \in \mathcal{P}_3(f) 
	\text{ and } \ell_f(v \cdot \unitvec{n}_f|_{f}) = 0
	\ \forall f \in \Delta_2(K), \\
	\curl v \in \GN(K_A), \text{ and }
	a_K(v, \grad \phi) = 0 
	\ \forall \phi \in V^{0, h}_0(K_A)  
	\}.
\end{multline}

\begin{lemma}
	\label{lem:h1curl-local-unisolvence}
	$\dim \tilde{V}^{1, h}(K_A) = 54$, the degrees of freedom in 
	\cref{eq:h1curl-dofs} with $\mesh = K_A$ are unisolvent, and 
	$\mathcal{P}_2(K)^3 \subset \tilde{V}^{1, h}(K_A)$.	
\end{lemma}
\begin{proof}
	Note that there are 54 linearly independent functionals 
	on $V^{1, h}(K_A)$ in \cref{eq:h1curl-dofs}. \\
	
	\noindent \textbf{Step 1: $\dim \tilde{V}^{1, h}(K_A) \leq 54$. } 
	Let $v \in \tilde{V}^{1, h}(K_A)$ and suppose that 
	the degrees of freedom in \cref{eq:h1curl-dofs} vanish. 
	Then, $\curl v$ vanishes on $\Delta_0(K)$. 
	Since $\curl v|_{e} \in \mathcal{P}_1(e)^3$ 
	for all edges $e \in \Delta_1(K)$, $\curl v$ also vanishes
	on $\Delta_1(K)$. Moreover, for $f \in \Delta_2(K)$,
	$\curl v \times \unitvec{n} \in \mathcal{P}_1(f)^3$,
	and so $\curl v \times \unitvec{n}|_f \equiv 0$.
	Finally,
	\begin{align*}
		\int_{f} \curl v \cdot \unitvec{n}_f \d{s} 
		= \sum_{e \in \Delta_1(f)} \mathcal{O}(e, f) 
		\int_{e} v \cdot \unitvec{t}_{e} \d{s}
		= 0,
	\end{align*}
	where we recall that $\mathcal{O}(e, f)$ is the orientation of $e$ relative 
	to $f$.
	Since $\curl v \cdot \unitvec{n}_f|_f \in \mathcal{P}_3(f) \cap H^1_0(f)$ 
	and additionally $\curl v\cdot \unitvec{n}_f|_{\partial f}
	= 0$, we have $\curl v \cdot \unitvec{n}_f|_{f} \equiv 0$.
	Thus, $\curl v|_{\partial K} \equiv 0$ and 
	so $\curl v \equiv 0$ since 
	$\curl v \in \GN(K_A)$.
	
	Thanks to the exactness of \cref{eq:stokes-complex-local-full-alfeld}, 
	there exists
	$\phi \in V^{0, h}(K_A)$ such that $v = \grad \phi$. 
	By the definition of $\tilde{V}^{1, h}(K_A)$, 
	$\phi \in \tilde{V}^{0, h}(K_A)$. 
	Moreover, since $\int_e v \cdot \unitvec{t}_e = 0$ for all 
	$e \in \Delta_1(K)$, 
	the fundamental theorem of calculus shows that $\phi$ takes the same
	value at every vertex in $\Delta_0(K)$. By subtracting a constant
	from $\phi$, which does not change $\grad \phi$, we may
	assume that $\phi$ vanishes at the vertices. Thus,
	\cref{eq:walkington-dofs} vanishes, and so $\phi \equiv 0$ by 
	\cref{lem:walkington-dofs-unisolvent}.
	Consequently, $v \equiv 0$ and 
	$\dim \tilde{V}^{1, h}(K_A) \leq 54$.
	
	\textbf{Step 2: $\dim \tilde{V}^{1, h}(K_A) \geq 54$. } 
	Suppose we are given
	arbitrary values for the degrees of freedom in \cref{eq:h1curl-dofs}
	$\{ c_{z, 1}^{\alpha}, \ldots, c_{z, 3}^{\alpha}, c_e
	: |\alpha| \leq 1, \ z \in \Delta_0(K), \ e \in \Delta_1(K) \}$.
	We now show that there exists $v \in \tilde{V}^{1, h}(K_A)$ with 
	these degrees of freedom. We do this using the degrees of freedom for 
	the space $V^{1, h}(K_A)$ given in 
	\cite[Lemma 4.11]{FuGuzmanNeilan20} (with slight modification).
	
	Let $w_c \in \GN(K_A)$ satisfy
	\begin{subequations}
		\label{eq:proof:gn-interp-curl}
		\begin{alignat}{2}
			w_c(z) &= \sum_{i, j, k=1}^{3} \epsilon_{ijk} c_{z, j}^{e_i} 
			\unitvec{e}_k \qquad
			& &\forall z \in \Delta_0(K), \\
			\int_{f} w_c \cdot \unitvec{n}_f \d{s} &= 
			\sum_{e \in \Delta_1(f)}  \mathcal{O}(e, f) c_e \qquad
			& &\forall f \in \Delta_2(K),
		\end{alignat}	
	\end{subequations}
	where $\epsilon_{ijk}$ is the permutation symbol. 
	Moreover, for $f \in \Delta_2(K)$, let $w_f \in \mathcal{P}_3(f)$ be 
	the unique cubic polynomial satisfying
	\begin{alignat*}{2}
		D_f^{\alpha} w_f(z) &= \sum_{i=1}^{3} c_{z, i}^{\alpha} 
		(\unitvec{n}_f \cdot \unitvec{e}_i) \qquad 
		& &\forall |\alpha| \leq 1, \ \forall z \in \Delta_0(f), \\
		\ell_f(w_f) \d{s} &= 0. \qquad & &
	\end{alignat*}
	Thanks to \cite[Lemma 4.11]{FuGuzmanNeilan20}, there exists 
	$v \in V^{1, h}(K_A)$ satisfying
	\begin{alignat*}{2}
		D^{\alpha} v(z) \cdot \unitvec{e}_i &= c_{z, i}^{\alpha} 
		\qquad & &\forall |\alpha| \leq 1, 
		\forall z \in \Delta_0(K), \ i \in 1:3, \\
		\int_{e} v \cdot \unitvec{t}_e \d{s} &= c_{e} 
		\qquad & &\forall e \in \Delta_1(K), \\
		\int_{e} v \cdot \unitvec{n}_f \d{s} &= \int_{e} w_f \d{s} 
		\qquad & & \forall e \in \Delta_1(f), \ 
		\forall f \in \Delta_2(K), \\
		\int_{e} \curl v \d{s} &= \int_{e} w_{c} \d{s} 
		\qquad & &\forall e \in \Delta_1(K), \\
		\int_{f} (v \cdot \unitvec{n}_f) \kappa \d{s} 
		&= \int_{f} w_f \kappa \d{s} 
		\qquad & &\forall \kappa \in \mathcal{P}_1(f), 
		\ \forall f \in \Delta_2(K), \\
		\int_f \curl v \times \unitvec{n}_f \d{s} 
		&= \int_f w_{c} \times \unitvec{n}_f \d{s}
		\qquad & &\forall f \in \Delta_2(K), \\
		a_K(v, \grad \phi) &= 0 
		\qquad & &\forall \phi \in V^{0, h}_0(K_A), \\
		\int_{K}  \curl v \cdot u \d{x} 
		&= \int_{K} w_c \cdot u \d{x}
		\qquad & &\forall u \in 
		\curl V^{1, h}_0(K_A).
	\end{alignat*} 
	By construction, we have
	\begin{alignat*}{2}
		D^{\alpha} v(z) \cdot \unitvec{e}_i &= c_{z, i}^{\alpha} 
		\qquad & & \forall|\alpha| \leq 1, \ \forall z \in \Delta_0(K), 
		\ i \in 1:3, \\
		\int_{e} v \cdot \unitvec{t}_e \d{s} &= c_e
		\qquad & & \forall e \in \Delta_1(K),
	\end{alignat*}
	so it remains to show that 
	$v \in \tilde{V}^{1, h}(K_A)$. We first show that 
	\begin{align*}
		v \cdot \unitvec{n}_f|_{f} \in \mathcal{P}_3(f) 
		\quad \text{and} \quad \ell_f(v \cdot \unitvec{n}_f|_{f}) = 0
		\qquad \forall f \in \Delta_2(K).
	\end{align*}
	Let $f \in \Delta_2(K)$. Note that by construction, we have 
	$u_f := v \cdot \unitvec{n}_f - w_f \in \mathcal{P}_{3}(f)$
	satisfies \cref{eq:proof:Pp1-face-dofs} with $p=3$, and so 
	$v \cdot \unitvec{n}_f = w_f$ on $f$.	 
	Analogous arguments show that $\curl v = w_c \in \GN(K_A)$. 
	Thus, $v \in \tilde{V}^{1, h}(K_A)$. \\	  
	
	\noindent \textbf{Step 3: Inclusion of quadratics. }
	Note that if $v \in \mathcal{P}_2(K)^3$, then
	$\ell_f(v \cdot \unitvec{n}_f) = 0$ for all $f \in \Delta_2(K)$
	by \cref{eq:kill-cubic-bubble-dof}, 
	$\curl v \in \mathcal{P}_1(K)^3 \subset \GN(K_A)$,
	and $a_K(v, \grad \phi) = 0$ for all $\phi \in V^{0, h}_0(K_A)$
	by \cref{eq:kill-h2-interior-bilinear-kill-p2}; therefore,
	$\mathcal{P}_2(K)^3 \subset \tilde{V}^{1, h}(K_A)$.
\end{proof}

\subsection{Global spaces and complexes}

We begin with a simple result showing that reducing the degree of 
normal components on faces also reduces the degree of the normal components 
on edges.
\begin{lemma}
	\label{lem:h2h1curl-edge-normal-degree-reduction}
	Let $K \in \mesh$. For $v \in \tilde{V}^{0, h}(K_A)$ and 
	$w \in \tilde{V}^{1, h}(K_A)$, there holds 
	\begin{align}
		\grad v|_e \times \unitvec{t}_e \in \mathcal{P}_3(e)
		\quad \text{and} \quad 
		w|_e \times \unitvec{t}_e \in \mathcal{P}_3(e) 
		\qquad \forall e \in \Delta_1(K).
	\end{align}
\end{lemma}
\begin{proof}
	Let $e \in \Delta_1(f) \cap \Delta_1(f')$ for 
	two distinct faces $f, f' \in \Delta_2(K)$.
	Since the normal derivatives 
	$\partial_{\unitvec{n}_f} v|_f$ and 
	$\partial_{\unitvec{n}_f} v|_{f'}$ are both cubic
	and $\{ \unitvec{n}_f, \unitvec{n}_{f'} \}$ form a basis for 
	$\mathbb{R}^3 \times \unitvec{t}_e$, we have
	$\grad v|_e \times \unitvec{t}_e \in \mathcal{P}_3(e)$.
	Similar arguments show that 
	$w|_e \times \unitvec{t}_e \in \mathcal{P}_3(e)$.
\end{proof}

We now proceed space-by-space in \cref{eq:stokes-complex-reduced-alfeld-bcs}
and show that the global spaces \cref{eq:min-h2h1curl-space} defined in terms
of the local spaces 
\cref{eq:local-h2-space-def,eq:local-h1curl-space-def}
properly glue together.

\begin{lemma}
	\label{lem:min-h2-space-dofs-bcs}
	The space $\tilde{V}^{0, h}_{\Gamma_0}$ is unisolvent 
	with respect to the degrees of freedom 
	\begin{subequations}
		\label{eq:min-h2-space-dofs-bcs}
		\begin{alignat}{2}
			&D^{\alpha} v(z) \qquad 
			& &\forall |\alpha| \leq 2, \ \forall z \in \Delta_0(\mesh) 
			\setminus \Delta_0(\boundarymesh), \\
			&\partial_{\unitvec{n}_{\Gamma}}^2 v(z) \qquad 
			& &\forall z \in \Delta_0^{\flat}(\boundarymesh),
		\end{alignat}
	\end{subequations}
	and $\dim \tilde{V}^{0, h}_{\Gamma_0} = 
	10(|\Delta_0(\mesh)| - |\Delta_0(\boundarymesh)|)
	+ |\Delta_0^{\flat}(\boundarymesh)|$.
\end{lemma}
\begin{proof}
	\textbf{Step 1: Upper bound. }
	Suppose that $v \in \tilde{V}^{0, h}_{\Gamma_0}$ and \cref{eq:min-h2-space-dofs-bcs}
	vanish. Thanks to \cref{lem:vertex-trace-bcs}, $D^{\alpha} v(z) = 0$
	for all $|\alpha| \leq 2$ and $z \in \Delta_0(\mesh)$. Thus,
	on each cell $K \in \mesh$, the degrees of freedom 
	\cref{eq:walkington-dofs} with $\mesh = K_A$ vanish, and so $v|_K \equiv 0$
	by \cref{lem:walkington-dofs-unisolvent}. Thus, 
	$\dim \tilde{V}^{0, h}_{\Gamma_0} \leq 
	10(|\Delta_0(\mesh)| - |\Delta_0(\boundarymesh)|)
	+ |\Delta_0^{\flat}(\boundarymesh)|$. \\
	
	\noindent \textbf{Step 2: Lower bound. } Suppose we are given values for 
	\cref{eq:min-h2-space-dofs-bcs}: 
	\begin{align}
		\label{eq:proof:min-h2-space-dofs-bcs}
		\{c_z^{\alpha}, d_{z'} : |\alpha| \leq 2, 
		\ z \in \Delta_0(\mesh) \setminus \Delta_0(\boundarymesh), 
		\ z' \in \Delta_0^{\flat}(\boundarymesh)\}.
	\end{align}
	For $K \in \mesh$, define $v_K \in \tilde{V}^{0, h}(K_A)$ by 
	\begin{alignat*}{2}
		D^{\alpha} v_K(z) &= c_z^{\alpha} \qquad 
		& &\forall |\alpha| \leq 2, \ \forall z \in \Delta_0(K)
		\setminus \Delta_0(\boundarymesh), \\
		\hess v_K(z') &= d_{z'} \unitvec{n}_{\Gamma} \otimes 
		\unitvec{n}_{\Gamma} 
		\qquad 
		& &\forall z' \in \Delta_0(K) \cap \Delta_0^{\flat}(\boundarymesh),
	\end{alignat*}
	with all remaining degrees of freedom in \cref{eq:walkington-dofs} with 
	$\mesh = K_A$ set to 0.
	Let $v \in L^2(\Omega)$ be defined by $v|_K := v_K$ for all 
	$K \in \mesh$. 
	
	We now show that $v \in \tilde{V}^{0, h}_{\Gamma_0}$.
	Suppose that $f \in \Delta_2(\mesh)$ satisfies either (i) there exist 
	distinct
	$K, K' \in \mesh$ with $f \in \Delta_2(K) \cap \Delta_2(K')$
	or (ii) $f \in \boundarymesh$ and there exists a 
	unique $K_f \in \mesh$ with $f \in \Delta_2(K)$. 
	Define $v_f \in \mathcal{P}_5(f)$
	and $u_f \in \mathcal{P}_3(f)$ by 
	(i) $v_f := v_K|_f - v_{K'}|_f$ and 
	$u_f := \partial_{\unitvec{n}_f} v_K|_f - \partial_{\unitvec{n}_f} v_K'|_f$
	or 
	(ii) $v_f := v_{K_f}|_f$ and $u_f := \partial_{\unitvec{n}_f} v_{K_f}|_f$. 
	
	In both cases, we have that $v_f$ belongs to the 
	2D Bell finite element space thanks to 
	\cref{lem:h2h1curl-edge-normal-degree-reduction}
	and $D_f^{\alpha} v_f(z) = 0$ for all $z \in \Delta_0(f)$.
	These degrees of freedom are unisolvent \cite[Theorem 2.2.12]{Ciarlet02},
	and so $v_f \equiv 0$. Also in both cases,
	\begin{align*}
		D_f^{\beta} u_f(z) = 0 \qquad \forall |\beta| \leq 1, 
		\ \forall z \in \Delta_0(f) 
		\quad \text{and} \quad 
		\ell_f(u_f) = 0.
	\end{align*}
	These degrees of freedom are clearly unisolvent on $\mathcal{P}_3(f)$,
	and so $u_f \equiv 0$. Thus, $v \in \tilde{V}^{0, h}_{\Gamma_0}$
	and the degrees of freedom of $v$ in \cref{eq:min-h2-space-dofs-bcs}
	match \cref{eq:proof:min-h2-space-dofs-bcs}, completing the proof.
\end{proof}

\begin{lemma}
	\label{lem:min-h1curl-space-dofs-bcs}
	The space $\tilde{V}^{1, h}_{\Gamma_0}$ is unisolvent with respect to the 
	degrees of freedom 
	\begin{subequations}
		\label{eq:min-h1curl-space-dofs-bcs}
		\begin{alignat}{2}
			&D^{\alpha} v(z) \qquad 
			& &\forall |\alpha| \leq 1, \ \forall z \in \Delta_0(\mesh) 
			\setminus \Delta_0(\boundarymesh), \\
			&\int_e v \cdot \unitvec{t}_e \d{s} \qquad 
			& &\forall e \in \Delta_1(\mesh) \setminus \Delta_1(\boundarymesh), 
			\\ 
			&\partial_{\unitvec{n}_{\Gamma}} (v \cdot \unitvec{n}_{\Gamma})(z) 
			\qquad 
			& &\forall z \in \Delta_0^{\flat}(\boundarymesh),
		\end{alignat}
	\end{subequations}
	and $\dim \tilde{V}^{1, h}_{\Gamma_0} = 
	12(|\Delta_0(\mesh)| - |\Delta_0(\boundarymesh)|)
	+ |\Delta_1(\mesh)| - |\Delta_1(\boundarymesh)|
	+ |\Delta_0^{\flat}(\boundarymesh)|$.
\end{lemma}
\begin{proof}
	\textbf{Step 1: Upper bound. }
	Suppose that $v \in \tilde{V}^{1, h}_{\Gamma_0}$ and 
	\cref{eq:min-h1curl-space-dofs-bcs}
	vanish. Thanks to \cref{lem:vertex-trace-bcs}, 
	on each cell $K \in \mesh$, the degrees of freedom 
	\cref{eq:h1curl-dofs} with $\mesh = K_A$ vanish, and so $v|_K \equiv 0$
	by \cref{lem:h1curl-local-unisolvence}. Thus, the dimension count
	in the statement of the lemma is an upper bound. \\
	
	\noindent \textbf{Step 2: Lower bound. } Suppose we are given values for 
	\cref{eq:min-h1curl-space-dofs-bcs}: 
	\begin{align}
		\label{eq:proof:min-h1curl-space-dofs-bcs}
		\{c_z^{\alpha}, c_e, d_{z'} : |\alpha| \leq 1, 
		\ z \in \Delta_0(\mesh) \setminus \Delta_0(\boundarymesh), 
		\ e \in \Delta_1(\mesh) \setminus \Delta_1(\boundarymesh),
		\ z' \in \Delta_0^{\flat}(\boundarymesh)\}.
	\end{align}
	For $K \in \mesh$, define $v_K \in \tilde{V}^{1, h}(K_A)$ by 
	\begin{alignat*}{2}
		D^{\alpha} v_K(z) &= c_z^{\alpha} \qquad 
		& &\forall |\alpha| \leq 1, \ \forall z \in \Delta_0(K)
		\setminus \Delta_0(\boundarymesh), \\
		\int_e v_K \cdot \unitvec{t}_e \d{s} &= c_e \qquad 
		& &\forall e \in \Delta_1(K)
		\setminus \Delta_1(\boundarymesh), \\
		\grad v_K(z') &= d_{z'} \unitvec{n}_{\Gamma} \otimes \unitvec{n}_{\Gamma}
		\qquad 
		& &\forall z' \in \Delta_0(K) \cap \Delta_0^{\flat}(\boundarymesh),
	\end{alignat*}
	with all remaining degrees of freedom  
	in \cref{eq:h1curl-dofs} with $\mesh = K_A$ set to 0.
	Let $v \in L^2(\Omega)^3$ be defined by $v|_K := v_K$ for all 
	$K \in \mesh$. 
	
	We now show that $v \in \tilde{V}^{1, h}_{\Gamma_0}$.
	Suppose that $f \in \Delta_2(\mesh)$ and $K, K', K_f \in \mesh$ are as in 
	Step 2 of the proof of \cref{lem:min-h2-space-dofs-bcs}. 
	Define $v_f \in \mathcal{P}_4(f)^3$
	and $u_f \in \mathcal{P}_3(f)^3$ by (i) $v_f := v_K|_f - v_{K'}|_f$ and 
	$u_f := \curl v_K|_f - \curl v_K'|_f$ or 
	(ii) $v_f := v_{K_f}|_f$ and $u_f := \curl v_{K_f}|_f$. We treat both 
	cases simultaneously.
	
	Since $\curl v \in \GN(K_A)$,
	$u_f \in \mathcal{P}_1(f)^3 \oplus \spann\{ b_f \unitvec{n}_f \}$,
	where we recall that $b_f$ is the cubic bubble 
	function on $f$. Clearly, we have $\rot_f v = u_f \cdot \unitvec{n}_f$,
	and so
	\begin{align*}
		u_f(z) = 0 \qquad \forall z \in \Delta_0(f)
		\quad \text{and} \quad 
		\int_f u_f \cdot \unitvec{n}_f \d{s} 
		= \sum_{e \in \Delta_1(f)} \mathcal{O}(e, f) 
		\int_{e} v \cdot \unitvec{t}_{e} \d{s} = 0,
	\end{align*}
	which are unisolvent on 
	$\mathcal{P}_1(f)^3 \oplus \spann\{ b_f \unitvec{n}_f \}$, and so 
	$u_f \equiv 0$.
	
	Turning to $v_f$, we have $v_f \cdot \unitvec{n}_f \in \mathcal{P}_3(f)$
	by definition. Since 
	\begin{align*}
		D_f^{\alpha} (v_f \cdot \unitvec{n}_f)(z) = 0
		\qquad \forall |\alpha| \leq 1, \ \forall z \in \Delta_0(f) 
		\quad \text{and} \quad 
		\ell_f(v_f \cdot \unitvec{n}_f) = 0
	\end{align*}
	and the above degrees of freedom are unisolvent on $\mathcal{P}_3(f)$,
	$v_f \cdot \unitvec{n}_f \equiv 0$. 
	
	We now show that $v|_{\partial f} \equiv 0$.
	As $\rot_f v_f \equiv 0$ and 
	$v_f \cdot \unitvec{n}_f \equiv 0$, de Rham's theorem 
	(see e.g. \cite[p. 31, Theorem 2.9]{GiraultRaviart86}) 
	shows that there exists $\phi \in H^1(f)$ such that 
	$v_f = \grad_f \phi$. Clearly, $\phi \in \mathcal{P}_5(f)$,
	and more specifically, $\phi$ belongs to the Bell 
	finite element space thanks to 
	\cref{lem:h2h1curl-edge-normal-degree-reduction}. 
	Arguing as in the proof of 
	\cref{lem:h1curl-local-unisolvence}, we may choose $\phi$ to vanish 
	at the vertices of $f$, and so all derivatives up to order 2
	of $\phi$ vanish on $\Delta_0(f)$. Arguing as in the 
	proof of \cref{lem:min-h2-space-dofs-bcs}, we have $\phi \equiv 0$
	and so $v \equiv 0$.
	Thus, $v \in \tilde{V}^{1, h}_{\Gamma_0}$
	and the degrees of freedom of $v$ in \cref{eq:min-h1curl-space-dofs-bcs}
	match \cref{eq:min-h1curl-space-dofs-bcs}. As a result, 
	the dimension count
	in the statement of the lemma is a lower bound.	
\end{proof}

Standard arguments also show that $\tilde{V}_{\Gamma_0}^{2, h}$ may be
characterized similarly by omitting degrees of freedom associated to 
$\Delta(\boundarymesh)$:
\begin{lemma}
	\label{lem:min-h1-space-dofs-bcs}
	The space $\tilde{V}^{2, h}_{\Gamma_0}$ is unisolvent with respect to the 
	degrees of freedom 
	\begin{subequations}
		\label{eq:min-h1-space-dofs-bcs}
		\begin{alignat}{2}
			&v(z) \qquad 
			& &\forall |\alpha| \leq 1, \ \forall z \in \Delta_0(\mesh) 
			\setminus \Delta_0(\boundarymesh), \\
			&\int_f v \cdot \unitvec{n}_f \d{s} \qquad 
			& &\forall e \in \Delta_2(\mesh) \setminus \Delta_2(\boundarymesh), 
		\end{alignat}
	\end{subequations}
	and $\dim \tilde{V}^{2, h}_{\Gamma_0} = 
	3(|\Delta_0(\mesh)| - |\Delta_0(\boundarymesh)|)
	+ |\Delta_2(\mesh)| - |\Delta_2(\boundarymesh)|$.
\end{lemma}

\begin{remark}
	\label{rem:remove-dofs-to-get-bcs}
	A simple consequence of \cref{lem:min-h2-space-dofs-bcs,%
		lem:min-h1curl-space-dofs-bcs,lem:min-h1-space-dofs-bcs}
	and their proofs is that $v \in \tilde{V}^{k, h}$ belongs 
	to $\tilde{V}^{k, h}_{\Gamma_0}$, $k \in 0:2$, if and only if 
	\begin{alignat*}{2}
		D^{\alpha} v(z) &= 0 \qquad & &\forall |\alpha| \leq 2-k, 
		\ \forall z \in \Delta_0(\boundarymesh) 
		\setminus \Delta_0^{\flat}(\boundarymesh), \\
		D^{\alpha} v(z) &= 0 \qquad & &\forall |\alpha| \leq \max\{ 1-k, 0\}, 
		\ \forall z \in \Delta_0^{\flat}(\boundarymesh), \\
		(I - \unitvec{n}_{\Gamma} \otimes \unitvec{n}_{\Gamma}) \grad^{2-k} v(z)
		&= 0 \qquad & & \forall z \in \Delta_0^{\flat}(\boundarymesh), \\
		\mathcal{I}^k_{\tau}(v) &= 0 \qquad 
		& &\forall \tau \in \Delta_k(\boundarymesh),
	\end{alignat*}
	where $\grad^2 := \hess$, $\grad^0 := I$, and we recall that 
	$\mathcal{I}^k_{\tau}$ is defined in \cref{eq:whitney-dofs-currents}.
\end{remark}

\section{Cohomology of the ``minimal'' complexes}
\label{sec:cohomology-reduced}

We could proceed as in \cref{sec:cohomology-full-poly-space} and 
apply the framework in \cite{HuLiangLin25}. However, the degrees of 
freedom for the discrete spaces $\tilde{V}^{k, h}$ are simple enough 
that we can prove the key results directly. To this end, we define 
the skeletal spaces for $k \in 0:2$ by 
\begin{align*}
	\tilde{\mathcal{S}}^{k}_{\Gamma_0} &:= 
	\left\{ v \in \tilde{V}^{k, h}_{\Gamma_0} : D^{\alpha} v(z) = 0 
	\ \forall |\alpha| \in \max\{1-k, 0\} : 2-k, 
	\ \forall z \in \Delta_0(\mesh) \right\},
\end{align*}
and set $\tilde{\mathcal{S}}^3_{\Gamma_0} := \tilde{V}^{3, h}_{\Gamma_0}$. Owing 
to \cref{lem:min-h2-space-dofs-bcs,lem:min-h1curl-space-dofs-bcs,%
	lem:min-h1-space-dofs-bcs}, the ``currents" 
$\mathcal{I}^k_{\tau}(\cdot)$ for
${\tau \in \Delta_k(\mesh) \setminus \Delta_k(\boundarymesh)}$
defined in \cref{eq:whitney-dofs-currents} are a unisolvent set of 
degrees of freedom on $\tilde{\mathcal{S}}^{k}_{\Gamma_0}$.

We also define the lifted (modified) bubble spaces associated 
with each vertex $z \in \Delta_0(z)$ for $k \in 0:2$ by
\begin{multline*}
	\tilde{\mathbb{B}}^{k}_{\Gamma_0}(z) := 
	\left\{ v \in \tilde{V}^{k, h}_{\Gamma_0} :
	D^{\alpha} v(z') = 0 \ \forall |\alpha| \leq 2-k, 
	\ \forall z' \in \Delta_0(\mesh) \setminus \{z\} \right. \\
	\left. 
	\vphantom{v \in \tilde{V}^{k, h}_{\Gamma_0}}
	\text{ and } \mathcal{I}_{\tau}^k(v) = 0 
	\ \forall \tau \in \Delta_k(\mesh) \right\}.
\end{multline*}
Analogous to \cref{eq:stokes-complex-full-geom-decomp}, we have an 
associated diagram:
\begin{equation}
	\label{eq:stokes-complex-gn-geom-decomp}
	\begin{tikzcd}[ampersand replacement=\&]
		0 \arrow[r] 
		\& \tilde{V}^{0, h}_{\Gamma_0} \arrow[r, "\grad"] 
		\& \tilde{V}^{1, h}_{\Gamma_0} \arrow[r, "\curl"] 
		\& \tilde{V}^{2, h}_{\Gamma_0} \arrow[r, "\div"]
		\& \tilde{V}^{3, h}_{\Gamma_0} \arrow[r]
		\& 0 \\[-2em]
		\& \rotatebox{90}{=} \& \rotatebox{90}{=} \& \rotatebox{90}{=} \&
		\rotatebox{90}{=} \& \\[-2em]
		0 \arrow[r] 
		\& \tilde{\mathcal{S}}^{0}_{\Gamma_0} \arrow[r, "\grad"] 
		\& \tilde{\mathcal{S}}^{1}_{\Gamma_0} \arrow[r, "\curl"] 
		\& \tilde{\mathcal{S}}^{2}_{\Gamma_0} \arrow[r, "\div"]
		\& \tilde{\mathcal{S}}^{3}_{\Gamma_0} \arrow[r]
		\& 0 \\[-2em]
		\& \bigoplus\limits_{z \in \Delta_0(\mesh)} 
		\& \bigoplus\limits_{z \in \Delta_0(\mesh)} 
		\& \bigoplus\limits_{z \in \Delta_0(\mesh)} 
		\& 
		\& \\[-2em]
		0 \arrow[r] 
		\& \tilde{\mathbb{B}}^{0}_{\Gamma_0}(z) \arrow[r, "\grad"] 
		\& \tilde{\mathbb{B}}^{1}_{\Gamma_0}(z) \arrow[r, "\curl"] 
		\& \tilde{\mathbb{B}}^{2}_{\Gamma_0}(z) \arrow[r, "\div"]
		\& 0 \arrow[r]
		\& 0.
	\end{tikzcd}
\end{equation}
Each column is a direct sum decomposition thanks to 
\cref{lem:min-h2-space-dofs-bcs,lem:min-h1curl-space-dofs-bcs,%
	lem:min-h1-space-dofs-bcs}. The complex properties are summarized as follows.
\begin{lemma}
	Each row of \cref{eq:stokes-complex-gn-geom-decomp} is a complex,
	and the final row is exact. Moreover, the cohomologies of the first
	two rows are isomorphic.
\end{lemma}
\begin{proof}
	Note that the degrees of freedom satisfy the following property:
	If $v \in \tilde{V}_{\Gamma_0}^{h, k}$, $k \in 0:2$, then 
	\begin{align*}
		\mathcal{I}^{k}_{\tau}(v) = 0 \qquad \forall \tau \in \Delta_{k}(\mesh)
		\implies \mathcal{I}^{k+1}_{\eta}(\dee^k v) = 0 
		\qquad \forall \eta \in \Delta_{k+1}(\mesh),
	\end{align*}
	and similarly for any $z \in \Delta_0(\mesh)$, there holds 
	\begin{align*}
		D^{\alpha} v(z) = 0 \qquad \forall |\alpha| \in \max\{1-k, 0\}:2-k
		\implies D^{\beta} \dee^k v(z) = 0 \qquad \forall |\beta| \leq 1-k. 
	\end{align*}
	Thus, each row of \cref{eq:stokes-complex-gn-geom-decomp} is a complex.
	\Cref{lem:min-h2-space-dofs-bcs,lem:min-h1curl-space-dofs-bcs,%
		lem:min-h1-space-dofs-bcs} show that 
	\begin{align*}
		\dim \tilde{\mathbb{B}}^{1}_{\Gamma_0} = \dim \tilde{\mathbb{B}}^{0}_{\Gamma_0}
		+ \dim \tilde{\mathbb{B}}^{2}_{\Gamma_0},
	\end{align*}
	and so exactness of the final row follows from standard arguments. 
	Consequently, the cohomologies of the first two rows of 
	\cref{eq:stokes-complex-gn-geom-decomp}
	are isomorphic.
\end{proof}

\subsection{Proof of \cref{thm:cohomology-reduced-alfeld-bcs}}
\label{sec:proof-cohomology-reduced-alfeld-bcs}

For $k \in 0:3$, let 
\begin{align*}
	W^{k, h} := \{ v \in L^2(\Omega) \otimes \mathbb{X}^{k} : 
	\dee^k v \in L^2(\Omega) \otimes \mathbb{X}^{k+1}  
	\text{ and }
	v|_{K} \in W^{k, h}(K) \ \forall K \in \mesh\}
\end{align*}
denote the Whitney forms on $\mesh$
equipped with the canonical degrees of freedom $\mathcal{I}^k_{\tau}$, 
$\tau \in \Delta_k(\mesh)$. Set $W_{\Gamma_0}^{k, h} := W^{k, h} \cap W^{k}_{\Gamma_0}$
so that $\{ \mathcal{I}^k_{\tau} : \tau \in \Delta_k(\mesh) 
\setminus \Delta_k(\boundarymesh)\}$,
are a unisolvent set of degrees of freedom on $W_{\Gamma_0}^{k, h}$.
Then, $\tilde{\mathcal{S}}_{\Gamma_0}^{k}$ and $W_{\Gamma_0}^{k, h}$
are isomorphic, and let 
$\pi^k : \tilde{\mathcal{S}}_{\Gamma_0}^{k} \to W_{\Gamma_0}^{k, h}$
denote the isomorphism that maps an element of $\tilde{\mathcal{S}}_{\Gamma_0}^{k}$
to the unique element in $W_{\Gamma_0}^{k, h}$ with the same degrees of freedom.
The generalized Stokes theorem \cref{eq:currents-stokes-thm} then shows 
that the following diagram commutes:
\begin{equation*}
	\begin{tikzcd}[ampersand replacement=\&]
		0 \arrow[r] \arrow[d]
		\& \tilde{\mathcal{S}}_{\Gamma_0}^0 \arrow[r, "\grad"] \arrow[d, "\pi^0"] 
		\& \tilde{\mathcal{S}}_{\Gamma_0}^1 \arrow[r, "\curl"] \arrow[d, "\pi^1"]
		\& \tilde{\mathcal{S}}_{\Gamma_0}^2 \arrow[r, "\div"] \arrow[d, "\pi^2"]
		\& \tilde{\mathcal{S}}_{\Gamma_0}^3 \arrow[r] \arrow[d, "\pi^3"]
		\& 0 \arrow[d] \\
		0 \arrow[r] 
		\& W_{\Gamma_0}^{0, h} \arrow[r, "\grad"] 
		\& W_{\Gamma_0}^{1, h} \arrow[r, "\curl"] 
		\& W_{\Gamma_0}^{2, h} \arrow[r, "\div"]
		\& W_{\Gamma_0}^{3, h} \arrow[r]
		\& 0.
	\end{tikzcd}
\end{equation*}
Thus, the two sequences have isomorphic cohomologies, the second of 
which is isomorphic to \cref{eq:de-rham-complex-bcs-feec} 
\cite[Theorem 1 \& Corollary 2]{Licht17} (see \cite[Example 9]{Licht17}
for the precise application of \cite[Corollary 2]{Licht17}),
which in turn is isomorphic to \cref{eq:stokes-complex-bcs-feec} thanks 
to \cref{eq:harmonic-forms-same-dim-diff-smoothness}.

The exactness of \cref{eq:stokes-complex-reduced-alfeld-element}
follows on taking $\mesh = K_A$ and $\Gamma_0 = \emptyset$ and 
noting that 
$\ker (\grad : \tilde{V}^{0, h}(K_A) \to \tilde{V}^{1, h}(K_A)) = \mathbb{R}$ 
and $b_{k}(K) = 0$ for $k \in 1:3$ since $K$ is contractible.
\hfill \qedsymbol

\begin{remark}
	We could have connected the skeletal complex to the 
	relative simplicial cochain complex as we did in 
	\cref{sec:proof-cohomology-full-spaces-bcs} for 
	the full polynomial complex \cref{eq:stokes-complex-full-alfeld-bcs}.
	Here, we highlight an alternative approach which leverages existing 
	cohomology results for simpler finite element spaces.
\end{remark}

\section{Proof of \cref{thm:bounded-cochain-projections}}
\label{sec:proof-bounded-cochain-projections}

For a face $f\in \Delta_2(\mesh)$
and positive integer $p \in \mathbb{N}_0$, 
denote the 
$L^2(f)$-orthogonal projection operator onto 
$\mathcal{P}_p(f) \otimes \mathbb{X}^k$ by 
$\mathbb{P}^k_{p, f} : L^1(f) \otimes \mathbb{X}^k 
\to \mathcal{P}_p(f) \otimes \mathbb{X}^k$. Similarly, for 
$K \in \mesh$, let 
$\mathbb{P}^k_{p, K} : L^1(K) \otimes \mathbb{X}^k \to V^{k, h}(K_A)$
denote the $L^2(K)$-orthogonal projection operator onto $V^{k, h}(K_A)$. \\

\noindent \textbf{Step 1: Construction. }
Let $z \in \Delta_0(\mesh)$. We choose $f_{z} \in \Delta_2(\mesh)$ with 
$z \subsimplex f_{z}$, $\tau_{z} \in \Delta_2(\mesh) \cup \mesh$
with $z \subsimplex \tau_{z}$, and a basis for $\mathbb{R}^3$ 
$\{ \unitvec{\nu}_{z, j} \}_{j=1}^{3}$ 
as follows:
\begin{enumerate}
	\item[(a)] If $z \in \Delta_0^{\flat}(\boundarymesh)$, let $f_{z} \in
	\Delta_2(\boundarymesh)$, $\tau_z \in \mesh$, $\unitvec{\nu}_{z, j}$,
	$j \in 1:2$, span the tangent plane of $\Gamma_0$ at $z$,
	and $\unitvec{\nu}_{z, 3}$ be the unit outward normal of $\partial \Omega$ 
	at $z$.
	
	\item[(b)] If $z \in \Delta_0(\boundarymesh) 
	\setminus \Delta_0^{\flat}(\boundarymesh)$, 
	let $f_{z}, \tau_{z} \in  \Delta_2(\boundarymesh)$ be not coplanar,
	$\unitvec{\nu}_{z, j}$, $j \in 1:2$ span the tangent plane of $f_z$,
	and $\unitvec{\nu}_{z, 3}$ be in the tangent plane of $\tau_z$
	so that $\{\unitvec{\nu}_{z, j}\}_{j=1}^{3}$ is a basis for $\mathbb{R}^3$. 
	
	\item[(c)] If $z \in \Delta_0(\mesh) \setminus \Delta_0(\boundarymesh)$, 
	let $f_z, \tau_z \in \Delta_2(\mesh)$ be not coplanar and let 
	$\{\unitvec{\nu}_{z, j}\}_{j=1}^{3}$ be as in (b).
\end{enumerate}
We define $S_z^k \in \mathbb{R}^{3\times 3}_{\sym}$, $k \in 0:1$, according to 
\begin{align*}
	S_z^k : \unitvec{\nu}_{z, i} \otimes \unitvec{\nu}_{z, j} &= 
	(\sym\grad_{f_{z}} \mathbb{P}_{4, f_{z}}^1 \grad^{1-k} v^k)(z) 
	: \unitvec{\nu}_{z, i} \otimes \unitvec{\nu}_{z, j}, \\
	S_z^k : \unitvec{\nu}_{z, 3} \otimes \unitvec{\nu}_{z, 3} &= 
	(\sym\grad_{\tau_z} \mathbb{P}_{4, \tau_{z}}^1 \grad^{1-k} v^k)(z) 
	: \unitvec{\nu}_{z, 3} \otimes \unitvec{\nu}_{z, 3}, 
\end{align*}
for all $i \in 1:3$ and $j \in 1:2$,
where we recall that $\grad^0 = I$,
$\grad_{f}$ for $f \in \Delta_2(\mesh)$ is the surface 
gradient (here taking values in $\mathbb{R}^{3 \times 3}$), and we set 
$\grad_K := \grad$ for $K \in \mesh$.

We define $\tilde{\Pi}^k v^k$, $k \in 0:3$, by assigning the degrees of freedom
in \cref{eq:gn-dofs,eq:h1curl-dofs,eq:walkington-dofs} as follows:
\begin{align}
	\label{eq:proof:fortin-current-dofs}
	\mathcal{I}_{\tau}^{k}( \tilde{\Pi}^k v^k) = \mathcal{I}_{\tau}^{k}(v^k)
	\qquad \forall \tau \in \Delta_k(\mesh),
\end{align}
where we note that the above conditions are well-defined since 
$H^2(\Omega)$ is continuously embedded into $C(\bar{\Omega})$
and  $v \mapsto \int_{e} v \cdot \unitvec{t}_e \d{s}$
is a continuous linear functional on $H^1(\curl; \Omega)$ 
\cite[pp. 284-285]{Hiptmair02}.
The remaining vertex degrees of freedom for $z \in \Delta_0(\mesh)$ are given by 
\begin{alignat*}{2}
	\grad \tilde{\Pi}^0 v^0(z) &= (\mathbb{P}_{4, f_z}^{1} \grad v^0)(z),
	\quad &
	\hess \tilde{\Pi}^0 v^0(z) &= S_z^{0}, \\
	\tilde{\Pi}^1 v^1(z) &= (\mathbb{P}_{4, f_z}^{1} v^1)(z),
	\quad & 
	\grad \tilde{\Pi}^1 v^1(z) &= S_z^{1} 
	+ \frac{1}{2} \mskw(\mathbb{P}_{3, f_z}^{2} \curl v^1)(z), \\
	& \quad &
	\tilde{\Pi}^2 v^2(z) &= (\mathbb{P}_{3, f_z}^{2} v^2)(z),
\end{alignat*}	
where 
\begin{align*}
	\mskw(u) := \begin{pmatrix}
		0 & -u_3 & u_2 \\
		u_3 & 0 & -u_1 \\
		-u_2 & u_1 & 0
	\end{pmatrix} \qquad \forall u \in \mathbb{R}^3.
\end{align*}
In particular, we have the formal identity 
$\grad v - (\grad v)^T  = \mskw(\curl v)$. We also note that above the 
conditions are well-defined by the trace theorem.
\\

\noindent \textbf{Step 2: Commutativity. } Thanks to 
\cref{eq:currents-stokes-thm,eq:proof:fortin-current-dofs}, we have 
\begin{align*}
	\mathcal{I}_{\tau}^{k+1}( \dee^k \tilde{\Pi}^k v^k) 
	= \sum_{\eta \in \Delta_{k}(\tau)} \mathcal{O}(\eta, \tau)
	\mathcal{I}_{\tau}^{k}(\tilde{\Pi}^k v^k)
	&= \sum_{\eta \in \Delta_{k}(\tau)} \mathcal{O}(\eta, \tau)
	\mathcal{I}_{\tau}^{k}(v^k) \\
	&= \mathcal{I}_{\tau}^{k+1}( \dee^k v^k) 
	=  \mathcal{I}_{\tau}^{k+1}( \tilde{\Pi}^{k+1} \dee^k v^k)
\end{align*}
for all $\tau \in \Delta_{k+1}(\mesh)$ and $k \in 0:2$.
Moreover, we easily see from the choice of vertex degrees of freedom that 
for all $z \in \Delta_0(\mesh)$, there holds
\begin{align*}
	(\grad^{\ell} \tilde{\Pi}^0 v^0)(z) 
	&= \grad^{\ell-1} (\tilde{\Pi}^1 \grad v^0)(z)
	\qquad \ell \in 1:2 \\
	(\curl \tilde{\Pi}^1 v^1)(z) &= \tilde{\Pi}^2 \curl v^1 (z),
\end{align*}
and so $\tilde{\Pi}^{k+1} \dee^k v^k = \dee^k \tilde{\Pi}^{k} v^k$, 
$k \in 0:2$. \\

\noindent \textbf{Step 3: Trace preservation and projection. } Suppose that 
$v^k \in V^k_{\Gamma_0}$. Then, we have $\mathcal{I}^k_{\tau}(\tilde{\Pi}^k v^k) = 0$
for all $\tau \in \Delta_k(\boundarymesh)$.
Moreover, the choice of vertex
degrees of freedom ensure that
\begin{align*}
	(I - \unitvec{n}_{\Gamma}(z) \otimes \unitvec{n}_{\Gamma}(z)) S_z^k = 0
	\text{ if } z \in \Delta_0^{\flat}(\boundarymesh)
	\quad \text{and} \quad 
	S^k_z = 0 \text{ if } z \in \Delta_0(\boundarymesh) 
	\setminus \Delta_0^{\flat}(\boundarymesh)
\end{align*} 
for $k \in 0:1$ and additionally 
\begin{align*}
	\grad \tilde{\Pi}^0 v^0(z) = \tilde{\Pi}^1 v^1(z) 
	= \tilde{\Pi}^1 \curl v^1(z) = \tilde{\Pi}^2 v^2(z) = 0
	\qquad \forall z \in \Delta_0(\boundarymesh).
\end{align*}
Thus, the degrees of freedom in \cref{rem:remove-dofs-to-get-bcs}
vanish, and so $\tilde{\Pi}^k v^k \in V^k_{\Gamma_0}$.

That $\tilde{\Pi}^k$ is a projection readily follows from the choice 
of degrees and ranges of the operators $\mathbb{P}^k_{p, f_z}$ and 
$\mathbb{P}^k_{p, \tau_z}$. \\

\noindent \textbf{Step 4: Continuity. } Continuity of the operators 
follow from standard scaling and approximation arguments 
(e.g. analogous arguments to proof of 
\cite[Theorem 3.14]{Hiptmair02} and \cite{ScottZhang90}) and are omitted 
for brevity. \hfill \qedsymbol

\bibliographystyle{amsplain}
\bibliography{references}

\end{document}